\documentclass[11pt,leqno]{article}
\usepackage{amsmath,amssymb,amsthm,latexsym}
\usepackage{indentfirst}
\usepackage{amsmath}
\newtheorem{Definition}{\bf \large Definition}[section]
\newtheorem{Theorem}{\bf \large Theorem}[section]
\newtheorem{PROPOSITION}{\bf \large Proposition}[section]
\newtheorem{Corollary}{\bf \large Corollary}[section]
\newtheorem{lemma}{\bf \large Lemma}[section]
\newtheorem{ex}{\bf \large Example}[section]
\newtheorem{Remark}{\bf \large Remark}[section]
\newtheorem{Lemma}{\bf \large Lemma}[section]

\title{\textbf{M\"{o}bius Homogeneous Hypersurfaces  in $\mathbb{S}^{n+1}$}}
\author {\normalsize {Tongzhu Li$^a$,~Xiang Ma$^b$,~Changping Wang$^c$,~Peng Wang$^d$}\\
\small{$^a$Department of Mathematics, Beijing Institute of
Technology, Beijing, China.}\\
\small{$^b$LMAM, School of Mathematical sciences, Peking
University, Beijing, China.}\\
\small{$^{c,d}$College of mathematics and computer Science, Fujian Normal University, Fuzhou, China}\\
\small{E-mail: $^a$litz@bit.edu.cn; $^b$maxiang@math.pku.edu.cn;}\\
\small{$^c$cpwang@fjnu.edu.cn; $^d$pengwang@fjnu.edu.cn.}}
\date{}

\begin{document}
\maketitle
\begin{abstract}
Let $\mathbb{M}(\mathbb{S}^{n+1})$ denote the M\"{o}bius transformation group of  the $(n+1)$-dimensional sphere
$\mathbb{S}^{n+1}$. A hypersurface $x:M^n\to \mathbb{S}^{n+1}$ is called a M\"{o}bius homogeneous hypersurface if there exists a subgroup $G$ of $\mathbb{M}(\mathbb{S}^{n+1})$ such that the orbit $G\cdot p=x(M^n), p\in x(M^n)$.
In this paper, the M\"{o}bius homogeneous hypersurfaces  are classified completely up to a  M\"{o}bius transformation of $\mathbb{S}^{n+1}$.
\end{abstract}

\vskip 1 cm
\section{Introduction}
A diffeomorphism $\phi: \mathbb{S}^{n+1}\to \mathbb{S}^{n+1}$ is said to be a \emph{M\"{o}bius transformation}, if $\phi$ takes the set of round
$n$-spheres into the set of round $n$-spheres. All M\"{o}bius transformations form a transformation group, which is called the \emph{M\"{o}bius
transformation group} of $\mathbb{S}^{n+1}$ and denoted by $ \mathbb{M}(\mathbb{S}^{n+1})$. It is well-known that, for $n\geq 2$, the M\"{o}bius transformation group $\mathbb{M}(\mathbb{S}^{n+1})$ of $\mathbb{S}^{n+1}$ coincides
with the \emph{conformal transformation group} of $\mathbb{S}^{n+1}$, denoted by  $\mathbb{C}(\mathbb{S}^{n+1})$, i.e.,
$\mathbb{M}(\mathbb{S}^{n+1})=\mathbb{C}(\mathbb{S}^{n+1})$.

A notable class of hypersurfaces in M\"{o}bius differential geometry is the M\"{o}bius homogeneous hypersurfaces in $\mathbb{S}^{n+1}$.  A hypersurface $x:M^n\to \mathbb{S}^{n+1}$ is called a \emph{M\"{o}bius homogeneous hypersurface} if the hypersurface is an orbit of a subgroup of the
M\"{o}bius transformation group of $\mathbb{S}^{n+1}$. Such hypersurfaces are the most symmetrical ones among all hypersurfaces in $\mathbb{S}^{n+1}$ (in the extrinsic sense).

Standard examples of M\"{o}bius homogeneous hypersurfaces in $\mathbb{S}^{n+1}$ are the images of  M\"{o}bius transformations of the isometric homogeneous hypersurfaces in $\mathbb{S}^{n+1}$. The  \emph{isometric homogeneous hypersurface} in $\mathbb{S}^{n+1}$ is an orbit of a subgroup of  the  isometric transformation group $\textup{O}(n+2)$ of $\mathbb{S}^{n+1}$. Clearly, $\textup{O}(n+2)\subseteq \mathbb{M}(\mathbb{S}^{n+1})$.
Isometric homogeneous hypersurfaces in $\mathbb{S}^{n+1}$ is a subclass of \emph{isoparametric hypersurfaces}, which have been systematically studied and classified (see, for instance, \cite{ca1},\cite{ce},\cite{chi},\cite{chi1},\cite{di},\cite{mi},\cite{mu1},\cite{mu2}).

The isometric homogeneous hypersurfaces in $\mathbb{R}^{n+1}$ and $\mathbb{H}^{n+1}$ are also
 M\"{o}bius homogeneous hypersurfaces in $\mathbb{S}^{n+1}$. The hyperbolic space is defined by
\[\mathbb{H}^{n+1}=\{(y_0,\vec{y}_1)\in R^+\times \mathbb{R}^{n+1}| \langle y,y\rangle=-y^2_0+\vec{y}_1\cdot \vec{y}_1=-1\}.\]
We can define the conformal map  $\tau:
\mathbb{H}^{n+1}\mapsto \mathbb{S}^{n+1}$ by
$$\tau(y)=(\frac{1}{y_0},\frac{\vec{y}_1}{y_0}),~~~y=(y_0,\vec{y}_1)\in
\mathbb{H}^{n+1}.$$
The inverse of the \emph{stereographic projection} $\sigma:\mathbb{R}^{n+1}\mapsto \mathbb{S}^{n+1}$ is defined by
$$\sigma(u)=(\frac{1-|u|^2}{1+|u|^2},\frac{2u}{1+|u|^2}).$$
The conformal maps $\sigma$ and $\tau$ embed a hypersurface of $\mathbb{R}^{n+1}$
or $\mathbb{H}^{n+1}$ into $\mathbb{S}^{n+1}$.
If $f:M^n\to \mathbb{R}^{n+1}$ (or $\mathbb{H}^{n+1}$) be an isometric homogeneous hypersurface, then the hypersurface $x=\sigma\circ f$ (or $\tau\circ f$) is a M\"{o}bius homogeneous hypersurface in $\mathbb{S}^{n+1}$.

Another class of examples come from the \emph{cone hypersurface} in $\mathbb{R}^{n+1}$ over any isometric homogeneous hypersurface in $\mathbb{S}^{m+1}~(m<n)$.
\begin{ex}\label{ex01}
Let $u:M^m\to \mathbb{S}^{m+1}$ be an immersed hypersurface. We define the cone hypersurface in $\mathbb{R}^{n+1}$ over $u$ as
$$f:\mathbb{R}^+\times \mathbb{R}^{n-m-1}\times M^m\to \mathbb{R}^{n+1},~~1\leq m\leq n-1,$$
$$f(t,y,p)=(y,tu(p)).$$
\end{ex}
These examples above come from homogeneous hypersurfaces in $\mathbb{S}^{n+1}$, $\mathbb{R}^{n+1}$ or $\mathbb{H}^{n+1}$.
But there are some examples of M\"{o}bius homogeneous hypersurfaces which can't be obtained in this way.
In \cite{s}, Sulanke constructed a M\"{o}bius homogeneous surface, which is a cylinder over a logarithmic spiral in $\mathbb{R}^2$, and he classified all M\"{o}bius homogeneous surfaces in $\mathbb{R}^3$. As a generalization, for M\"{o}bius homogeneous surfaces in $\mathbb{S}^{n+1}$, the third author and Zhenxiao Xie \cite{wx} classified them in $\mathbb{S}^4$; the second and the fourth author (together with Pedit) gave a classification for Willmore 2-spheres in \cite{ma}.

In \cite{lim}, we constructed a \emph{logarithmic spiral cylinder}, which is a high dimensional
version of Sulanke's example.
\begin{ex}\label{ex02}
Let $\gamma:I\to \mathbb{R}^2$ be the \emph{logarithmic spiral} in the Euclidean plane $\mathbb{R}^2$ given by
$$\gamma(s)=(e^{cs}\cos s,e^{cs}\sin s),~~~ c>0.$$
The \emph{logarithmic spiral cylinder} in $\mathbb{R}^{n+1}$ over $\gamma(s)$ is defined by
\begin{equation*}
\begin{split}
&f=(\gamma,id):I\times \mathbb{R}^{n-1}\mapsto \mathbb{R}^{n+1},\\
&f(s,y)=f(s,y_1,\cdots,y_{n-1})=(e^{cs}\cos s,e^{cs}\sin s,y_1,\cdots,y_{n-1})\in \mathbb{R}^{n+1}.
\end{split}
\end{equation*}
\end{ex}

In \cite{chenli}, the first author and his students classified the M\"{o}bius homogeneous hypersurfaces with one simple principal curvature in $\mathbb{S}^{n+1}$. In \cite{jl}, we pointed out that the compact M\"{o}bius homogeneous submanifolds
are M\"{o}bius equivalent to the isometric homogeneous submanifolds in $\mathbb{S}^{n+1}$. M\"{o}bius homogeneous hypersurfaces with small dimension or small number of distinct principal curvatures are relatively easy to classify (see \cite{LiT},\cite{lim},\cite{lim2},\cite{w1}).

In this paper, we classify all M\"{o}bius homogeneous hypersurfaces.
The Main Theorem is as follows.
\begin{Theorem}\label{th11}
Let $x: M^n\rightarrow \mathbb{S}^{n+1}$ be a M\"{o}bius homogeneous hypersurfaces.
Then $x$ is M\"{o}bius equivalent to one of
the following examples:\\
(1)  the isometric homogeneous hypersurfaces in $\mathbb{S}^{n+1}$;\\
(2)  the isometric homogeneous ones in $\mathbb{R}^{n+1}$;\\
(3)  the isometric homogeneous ones in $\mathbb{H}^{n+1}$;\\
(4)  the cone hypersurfaces over isometric homogeneous hypersurfaces in $\mathbb{S}^{m+1}(m< n)$;\\
(5)  the logarithmic spiral cylinders.
\end{Theorem}

In the previous works, the main tool we used is the structure equations. The homogeneity forced every geometrically well-defined function to be constant, hence facilitates the computations and the analysis of various cases. Notice that in the M\"{o}bius geometry, the invariants and the structure equations are much more complicated compared with the usual geometry of space forms. Moreover, when the number of distinct principal curvatures are large, these equations are too involved and intimidating. This is the main technical difficulties we faced before.

In this paper, we adopt a new approach by utilizing the subgroup and orbit geometry of the Lorentz group, which is well-known equivalent to the M\"{o}bius transformation group $\mathbb{M}(\mathbb{S}^{n+1})$ and the isometry group of the hyperbolic space $\mathbb{H}^{n+2}$. By the work of Takagi and Takahashi \cite{ta}, the subgroups of $\mathbb{M}(\mathbb{S}^{n+1})$ are classified into three classes (see Theorem \ref{4-1}) according to its orbit behavior in $\mathbb{H}^{n+2}$. This helps to show that the M\"{o}bius form vanishes, hence overcomes the main technical obstacle on the way to the final proof of the Vanishing Theorem \ref{the1}. Thus the M\"{o}bius homogeneous hypersurfaces are Dupin hypersurfaces with constant M\"{o}bius curvatures, which were already classified by the first and the third authors together with Jie Qing \cite{lim3}.

We organize the paper as follows.

In section~2, we review the basic theory and facts about the M\"{o}bius transformation group of $\mathbb{S}^{n+1}$ and some examples of M\"{o}bius homogeneous hypersurfaces. In section~3, we give M\"{o}bius invariants of hypersurfaces in $\mathbb{S}^{n+1}$. In
section~4, we give the proof of the Main Classification Theorem \ref{th11}. This depends on the Vanishing Theorem \ref{the1} which claims that the M\"{o}bius form vanishes under quite general conditions for a M\"{o}bius homogeneous hypersurface.

The proof of this key result is divided into two sections. In section~5 we recall the result of Takagi and Takahashi \cite{ta} on the subgroup of the Lorentz orthogonal group acting on $\mathbb{H}^{n+2}$ and consider three different cases according to its orbit behavior. In the first two cases, it is easy to show that the M\"{o}bius form vanishes. The third case (the \emph{hyperbolic case}) needs to combine the orbit information derived in section~5 and the structure equation analysis in section~6.

\section{Examples of M\"{o}bius homogeneous hypersurfaces.}
In this section, we review some facts about the M\"{o}bius transformation group of $\mathbb{S}^{n+1}$. For details we refer to \cite{ce}, or \cite{o}.  And we present these subgroups whose orbits are the M\"{o}bius homogeneous hypersurfaces as in Example~\ref{ex01} and Example~\ref{ex02}.

Let $\mathbb{R}^{n+2}$ denote the $(n+2)$-dimensional Euclidean space, and a dot $\cdot$ represents its inner product. The $(n+1)$-dimensional sphere is
$\mathbb{S}^{n+1}=\{x\in \mathbb{R}^{n+2}|x\cdot x=1\}$. The hypersphere $S_p(\rho)$ in $\mathbb{S}^{n+1}$ with center $p\in \mathbb{S}^{n+1}$ and
radius $\rho$, is  given by
$$S_p(\rho)=\{y\in\mathbb{S}^{n+1}|p\cdot y=\cos\rho\}, ~~0<\rho<\pi.$$

Let $ \mathbb{D}^{n+2}=\{x\in\mathbb{R}^{n+2}| x\cdot x\leq1\}$. Taking  $o\in \mathbb{R}^{n+2}$ such that $o\notin \mathbb{D}^{n+2}$, a line $l$ that passes through the point $o$ intersects the sphere $\mathbb{S}^{n+1}$
in two points ${p,q}$. Now we define the M\"{o}bius inversion $\Upsilon_o$ for the point $o\notin \mathbb{D}^{n+2}\subset\mathbb{R}^{n+2}$ as follows,
$$\Upsilon_o: \mathbb{S}^{n+1}\to \mathbb{S}^{n+1},~~~\Upsilon_o(p)=q.$$
Clearly, $\Upsilon_o\in \mathbb{M}(\mathbb{S}^{n+1})$. When the point $o$ is at infinity, the M\"{o}bius inversion is indeed a reflection $\Upsilon_o\in \textup{O}(n+2)$, thus an isometric transformation of $\mathbb{S}^{n+1}$.
\begin{PROPOSITION}\cite{ce}
The M\"{o}bius transformation group $\mathbb{M}(\mathbb{S}^{n+1})$ is generated by M\"{o}bius inversions $\Upsilon_o$.
\end{PROPOSITION}

Let $\mathbb{R}^{n+3}_1$ be the Lorentz space, i.e., $\mathbb{R}^{n+3}$ with the scalar
 product $\langle,\rangle$ defined by
\[\langle x,y\rangle=-x_0y_0+x_1y_1+\cdots+x_{n+2}y_{n+2}\] for
$x=(x_0,x_1,\cdots,x_{n+2}), y=(y_0,y_1,\cdots,y_{n+2})\in \mathbb{R}^{n+3}$.

Let $GL(\mathbb{R}^{n+3})$ be the set of invertible $(n+3)\times(n+3)$ matrix, then the Lorentz orthogonal group $\textup{O(n+2,1)}$ is defined by
$$\textup{O(n+2,1)}=\{T\in GL(\mathbb{R}^{n+3})~|~ TI_1T^t=I_1\},$$
where $T^t$ denotes the transpose of the matrix $T$ and $I_1=\left(\begin{matrix}-1&0\\
0&I\end{matrix}\right),$ and $I$ is the  $(n+2)\times(n+2)$ unit matrix.

The positive light cone is
$$\mathbb{C}^{n+2}_+=\{y=(y_0,\vec{y}_1)\in \mathbb{R}\times \mathbb{R}^{n+2}=\mathbb{R}^{n+3}_1~|~\langle y,y\rangle=0, ~y_0>0\},$$
and $\textup{O}^+(n+2,1)$ is the subgroup of $\textup{O(n+2,1)}$ defined by
$$\textup{O}^+(n+2,1)=\{T\in \textup{O(n+2,1)}~|~ T(C^{n+2}_+)=C^{n+2}_+\}.$$
\begin{PROPOSITION}[\cite{o}]
Let  $T=\left(\begin{matrix}w&u\\
v&Q\end{matrix}\right)\in \textup{O(n+2,1)}$,
where $Q$ is an $(n+2)\times(n+2)$ matrix. then
$T\in \textup{O}^+(n+2,1)$ if and only if $w>0.$
\end{PROPOSITION}
It is well-known that the subgroup $\textup{O}^+(n+2,1)$
is isomorphic to the M\"{o}bius transformation group $\mathbb{M}(\mathbb{S}^{n+1})$. In fact, for any $$
T=\left(\begin{matrix}w&u\\
v&Q\end{matrix}\right)\in \textup{O}^+(n+2,1),
$$
we can define the M\"{o}bius transformation $\Psi(T):\mathbb{S}^{n+1}\mapsto \mathbb{S}^{n+1}$ by
\[\Psi(T)(x)=\frac{Qx+v}{u x+w}, ~~~x=(x_1,\cdots,x_{n+2})^t\in \mathbb{S}^{n+1}.\]
Then the map $\Psi:\textup{O}^+(n+2,1)\mapsto \mathbb{M}(\mathbb{S}^{n+1})$ is a group isomorphism.

Let $Q\in \textup{O}(n+2)$ be an isometric transformation of $\mathbb{S}^{n+1}$, then $Q\in \mathbb{M}(\mathbb{S}^{n+1})$ and
$$\Psi^{-1}(Q)=\left(\begin{matrix}1&0\\
0&Q\end{matrix}\right)\in \textup{O}^+(n+2,1).$$ Thus $\Psi^{-1}(\textup{O}(n+2))\subset \textup{O}^+(n+2,1)$ is a subgroup.

The $(n+1)$-dimensional sphere
$\mathbb{S}^{n+1}$
 is diffeomorphic to the projective light cone $PC^{n+1}$,
$$PC^{n+1}=\{[z]\in PR^{n+3}~|~z\in\mathbb{C}^{n+2}_+\}.$$
The diffeomorphism $\Phi:\mathbb{S}^{n+1}\to PC^{n+1}$ is given by $$\Phi(x)=[y]=[(1,x)].$$
The group $\textup{O}^+(n+2,1)$ acts on $PC^{n+1}$ by $$T[p]=[Tp],~T\in \textup{O}^+(n+2,1),~~ [p]\in PC^{n+1}.$$

Let $x: M^n\mapsto \mathbb{S}^{n+1}$ be a hypersurface without umbilical points.
Let $II$ and $H$ be the second fundamental form and
the mean curvature of $x$, respectively.
The M\"{o}bius position vector of $x$
$$Y: M^n\mapsto \mathbb{C}^{n+2}_+$$  is defined by
$$Y=\rho(1,x), ~~~~~\rho^2=\frac{n}{n-1}(\|II\|^2-nH^2).$$
Thus for a hypersurface $x: M^n\mapsto \mathbb{S}^{n+1}$ without umbilical point, by the diffeomorphism $PC^{n+1}\approx \mathbb{S}^{n+1}$ we have $$x(M^n)=[Y(M^n)]\subset PC^{n+1}.$$
And the M\"{o}bius position vector of $x$ define a submanifold $Y(M^n)$ in $\mathbb{C}^{n+2}_+\subset\mathbb{R}^{n+3}_1.$
\begin{Theorem}(\cite{w})\label{theorem-1}
Two hypersurfaces $x,\tilde{x}: M^n\mapsto \mathbb{S}^{n+1}$ are M\"{o}bius
equivalent if and only if there exists $T\in \textup{O}^+(n+2,1)$  such that $\tilde{Y}=YT.$
\end{Theorem}

Let $f:M^{n}\rightarrow \mathbb{R}^{n+1}$ be a hypersurface without umbilical points, using the conformal map $\sigma$ one also defines the M\"{o}bius position vector of $f$.
Let $II$ and $H$ be the second fundamental form and the mean curvature of $f$, respectively.
As in \cite{lim3,w},
the M\"{o}bius position vector of $f$ can be defined by
$$Y=\rho(f)\left(\frac{1+|f|^2}{2},\frac{1-|f|^2}{2},f\right): M^n\rightarrow \mathbb{C}^{n+2}_+
$$
where $(\rho (f))^2=\frac{n}{n-1}(|II|^2-nH^2)$.

Let $G$ be a subgroup of $\mathbb{M}(\mathbb{S}^{n+1})$. For any point $p\in \mathbb{S}^{n+1}$, the orbit of $G$ through $p$ is
$$G\cdot p=\{\phi(p)|\phi\in G\}.$$

\begin{Definition}
A hypersurface $x:M^n\to \mathbb{S}^{n+1}$ is called a M\"{o}bius homogeneous hypersurface in $\mathbb{S}^{n+1}$ if there exists a subgroup $G\subset\mathbb{M}(\mathbb{S}^{n+1})$ such that the orbit $x(M^n)=G\cdot p,~ p\in x(M^n)$.
\end{Definition}
It is convenient to construct a M\"{o}bius homogeneous hypersurface by a subgroup of $\textup{O}^+(n+2,1)$, whose action on $\mathbb{R}^{n+3}_1$ is linear. By Theorem \ref{theorem-1}, we have the following proposition,
\begin{PROPOSITION}\label{x-y-1}
A hypersurface $x:M^n\to \mathbb{S}^{n+1}$ is a M\"{o}bius homogeneous hypersurface in $\mathbb{S}^{n+1}$ if and only if there exists a subgroup $G\subset\textup{O}^+(n+2,1)$ such that the orbit $G\cdot p=Y(M^n),p\in \mathbb{C}^{n+2}_+ $, that is,  $Y(M^n)$ is a homogeneous submanifold in $\mathbb{R}^{n+3}_1$.
\end{PROPOSITION}

Standard examples of M\"{o}bius homogeneous hypersurfaces in $\mathbb{S}^{n+1}$ are  the isometric homogeneous hypersurfaces in space forms.  An isometric homogeneous hypersurface in $\mathbb{R}^{n+1}$ is a hyperplane, or a hypersphere, or one of the following cylinders:
$$
f:\mathbb{R}^m\times \mathbb{S}^{n-m}(r)\to \mathbb{R}^{n+1},~~1\leq m\leq n-1.
$$
An isometric homogeneous hypersurface in $\mathbb{H}^{n+1}$ is a hypersphere, or a horosphere, a hyperbolic hyperplane, or one of the following hyperbolic cylinders:
$$
f:\mathbb{H}^m(\sqrt{1+r^2})\times \mathbb{S}^{n-m}(r)\to \mathbb{H}^{n+1},~~1\leq m\leq n-1.
$$

\begin{PROPOSITION}\label{pro-201}
The image of $\sigma$ of the cylinder $\sigma\circ f(\mathbb{R}^m\times \mathbb{S}^{n-m}(r)), 1\leq m\leq n-1,$ is a M\"{o}bius homogeneous hypersurfaces in $\mathbb{S}^{n+1}$.
\end{PROPOSITION}

\begin{proof}
Up to a similarity, we only consider the cylinder $\sigma\circ f(\mathbb{R}^m\times \mathbb{S}^{n-m}), 1\leq m\leq n-1.$
Let $u=(u_1,\cdots,u_m)\in \mathbb{R}^m$,
 \[G_r=\left\{\left(\begin{array}{cccccc}
1+\frac{|u|^2}{2}&-\frac{|u|^2}{2} &u_1 &\cdots&u_{m}&0\\
\frac{|u|^2}{2}&1-\frac{|u|^2}{2} &u_1 &\cdots&u_{m}&0\\
u_1&-u_1 &1 &\cdots&0&0\\
 \vdots&\vdots&\vdots &\ddots&\vdots&\vdots\\
u_{m}&-u_{m}&0&\cdots&1 &0\\
0&0&0&\cdots&0&\textup{O}(n-m+1)
\end{array}\right)\right\}.
\] Then $G_r$ is a subgroup of $\textup{O}^+(n+2,1)$ and $[Y(\mathbb{R}^m\times \mathbb{S}^{n-m}))]$ is the orbit of
$G_r$ acting on the point $p=(\underbrace{1,0,\cdots,0}_{m+2},1,0,\cdots,0)\in \mathbb{C}^{n+2}_+$.
\end{proof}

\begin{PROPOSITION}\label{pro-202}
The image of $\tau$ of the hyperbolic cylinder
$$
\tau\circ f(\mathbb{H}^m(\sqrt{1+r^2})\times \mathbb{S}^{n-m}(r)), 1\leq m\leq n-1
$$
is a M\"{o}bius homogeneous hypersurfaces in $\mathbb{S}^{n+1}$.
\end{PROPOSITION}

\begin{proof}
Let
\[ G_h=\left\{\left(\begin{array}{cccc}
\textup{O}^+(m,1)&0 & 0&0\\
 0&1&0 &0 \\
0 & 0&0 &\textup{O}(n-m+1)
\end{array}\right)\right\}.
\]
 Then $G_h$ is a subgroup of $\textup{O}^+(n+2,1)$ and $[Y(\mathbb{H}^m(\sqrt{1+r^2})\times \mathbb{S}^{n-m}(r))]$ is the orbit of
$G_h$ acting on the point $p=(\underbrace{1,0,\cdots,0}_{m+2},1,0,\cdots,0)\in \mathbb{C}^{n+2}_+$.
\end{proof}

The isometric homogeneous hypersurfaces in $\mathbb{S}^{n+1}$ are M\"{o}bius homogeneous.
Due to Hsiang-Lawson \cite{Xiang} and Takagi-Takahashi \cite{ta}, the isometric homogeneous  hypersurfaces in $\mathbb{S}^{n+1}$ are classified.

\begin{Theorem}\label{isom}\cite{Xiang,ta}
Every homogeneous hypersurface in $\mathbb{S}^{n+1}$ can be
obtained as a principal orbit of a linear isotropy representation of a compact Riemannian symmetric
pair $(U,K)$ of rank two.
\end{Theorem}
The homogeneous hypersurfaces in $\mathbb{S}^{n+1}$ are always isoparametric. On the basis of Cartan's  \cite{ca1} and  M\"{u}zner's results \cite{mu1,mu2}, combined with the important work of many geometers, the isoparametric hypersurfaces have been classified (for example,see \cite{chi,chi1,mi}).

Next we prove that these hypersurfaces given by Example \ref{ex01} and Example \ref{ex02} are M\"{o}bius homogeneous.
\begin{PROPOSITION}\label{pro-21}
Let $u:M^m\to \mathbb{S}^{m+1}$ be a homogeneous hypersurface, then the cone hypersurface in $\mathbb{R}^{n+1}$ over $u$  $$\sigma\circ f: \mathbb{R}^+\times \mathbb{R}^{n-m-1}\times M^m\to \mathbb{S}^{n+1}$$ is a M\"{o}bius homogeneous hypersurface in $\mathbb{S}^{n+1}$.
\end{PROPOSITION}

\begin{proof}
Let $u:M^m\to \mathbb{S}^{m+1}$ be a homogeneous hypersurface in $(m+1)-$dimensional sphere. Then there exists a subgroup $G_1\subset \textup{O}(m+2)$ whose orbit is $u(M^m)$.

 Let \[ G_c=\left\{\left(\begin{array}{cccc}
\textup{O}^+(n-m,1)& & &0\\
 0& & &G_1
\end{array}\right)\right\},
\]  then $G_c$ is a subgroup of $\textup{O}^+(n+2,1)$.

The cone hypersurface $\sigma\circ f: \mathbb{R}^+\times \mathbb{R}^{n-m-1}\times M^m\to \mathbb{S}^{n+1}$ is the  orbit of the subgroup $\Psi(G_c)\subset \mathbb{M}(\mathbb{S}^{n+1})$
acting on the point $$p=(\underbrace{0,,0,\cdots,\frac{1}{\sqrt{2}}}_{n-m},\underbrace{\frac{1}{\sqrt{2}},0,\cdots,0}_{m+2})\in S^{n+1}.$$
Thus the cone hypersurface in $\mathbb{R}^{n+1}$ over $u$  is  M\"{o}bius homogeneous.
\end{proof}

\begin{PROPOSITION}\label{pro-22}
 Let $f=(\gamma,id):I\times \mathbb{R}^{n-1}\mapsto \mathbb{R}^{n+1}$ be a logarithmic spiral cylinder in $\mathbb{R}^{n+1}$,
then the hypersurface $\sigma\circ f$ is a M\"{o}bius homogeneous hypersurface in $\mathbb{S}^{n+1}$.
\end{PROPOSITION}

\begin{proof}
Let $y=(y_1,\cdots,y_{n-1})\in \mathbb{R}^{n-1}$,
\begin{equation}\label{group}
G_e=\left\{\left(\begin{array}{ccccccc}
\frac{1+|y|^2+e^{2cs}}{2e^{cs}}&\frac{1+|y|^2-e^{2cs}}{2e^{cs}} &0 &0&y_1&\cdots&y_{n-1}\\
\frac{1-|y|^2-e^{2cs}}{2e^{cs}}&\frac{1-|y|^2+e^{2cs}}{2e^{cs}} &0&0&-y_1&\cdots&-y_{n-1}\\
0&0&\cos s&-\sin s&0&\cdots&0\\
0&0&\sin s&\cos s&0&\cdots& 0\\
\frac{y_1}{e^{cs}}&\frac{y_1}{e^{cs}}&0&0&1&\cdots&0\\
\vdots&\vdots&\vdots&\vdots&\vdots&\ddots&\vdots\\
\frac{y_{n-1}}{e^{cs}}&\frac{y_{n-1}}{e^{cs}}&0&0&0&\cdots&1
\end{array}\right)\right\},
\end{equation}
then  $G_e$ is a subgroup of $\textup{O}^+(n+2,1)$.

Then the logarithmic spiral cylinder $\sigma\circ f$ is the orbit of the subgroup $\Psi(G_e)\subset M\ddot{o}b(\mathbb{S}^{n+1})$
acting on the point $p=(1,0,\cdots,0)\in \mathbb{S}^{n+1}.$
\end{proof}

\begin{Remark}
Here we point out that we also give the subgroup corresponding to logarithmic spiral cylinder in \cite{lim}, but there are typos in \cite{lim}. Now we revise the subgroup by (\ref{group}) and correct the mistakes in \cite{lim}.
\end{Remark}

\section{M\"{o}bius invariants for hypersurfaces in $\mathbb{S}^{n+1}$}
In this section, we recall M\"{o}bius invariants of hypersurfaces in $\mathbb{S}^{n+1}$.
For details we refer to \cite{w} or \cite{lim3}.

Let $x: M^n\mapsto \mathbb{S}^{n+1}$ be a hypersurface without umbilical points, and $e_{n+1}$ the unit normal vector field.
Let $II$ and $H$ be the second fundamental form and
the mean curvature of $x$, respectively.
The M\"{o}bius position vector $Y$ of $x$ is
$$Y: M^n\mapsto \mathbb{C}^{n+2}_+,$$
where $Y=\rho(1,x), ~~~\rho^2=\frac{n}{n-1}(\|II\|^2-nH^2).$

It follows immediately from Theorem 2.1 that
$$g=\langle dY,dY\rangle=\rho^2dx\cdot dx$$
is a M\"{o}bius invariant, which is called the M\"{o}bius metric of $x$ (see \cite{w}).

Let $\Delta$ be the Laplacian operator with respect to $g$, we
define
$$N=-\frac{1}{n}\Delta Y-\frac{1}{2n^2}\langle\Delta Y,\Delta Y\rangle Y.$$
Then we have
$$\langle Y,Y\rangle=0, ~~\langle N,Y\rangle=1,~~ \langle N,N\rangle=0.$$

Let $\{E_1,\cdots,E_n\}$ be a local orthonormal frame for
$(M^n,g)$ with the dual frame $\{\omega_1,\cdots,\omega_n\}$, and
write $Y_i=E_i(Y)$, then we have
$$\langle Y_i,Y\rangle=\langle Y_i,N\rangle=0,~~ \langle Y_i,Y_j\rangle=\delta_{ij},~~ 1\leq i,j\leq n.$$
We define the conformal Gauss map $\xi$  of $x$ by
$$\xi=(H,Hx+e_{n+1}).$$
By direct computations, we have
$$\langle \xi,Y\rangle=\langle \xi,N\rangle=\langle \xi,Y_i\rangle=0, ~~\langle \xi,\xi\rangle=1.$$

Then $\{Y,N,Y_1,\cdots,Y_n,\xi\}$ forms a moving frame in
$\mathbb{R}^{n+3}_1$ along $M^n$. We  use
the following range of indices in this section: $1\leq
i,j,k,l\leq n$. We can write the structure equations as follows:
\begin{eqnarray*}
&&dY=\sum_iY_i\omega_i,\label{con}\\
&&dN=\sum_{ij}A_{ij}\omega_iY_j+\sum_iC_i\omega_i\xi,\\
&&dY_i=-\sum_jA_{ij}\omega_jY-\omega_iN+\sum_j\omega_{ij}Y_j+\sum_jB_{ij}\omega_j\xi,\label{con2}\\
&&d\xi=-\sum_iC_i\omega_iY-\sum_{ij}\omega_jB_{ij}Y_i,\label{con1}
\end{eqnarray*}
where $\omega_{ij}$ is the connection form of the M\"{o}bius metric $g$
and $\omega_{ij}+\omega_{ji}=0$. The tensors $A=\sum_{ij}A_{ij}\omega_i\otimes\omega_j$, $C=\sum_iC_i\omega_i$
and $B=\sum_{ij}B_{ij}\omega_i\otimes\omega_j$ are called the
Blaschke tensor, the M\"{o}bius form and the M\"{o}bius second fundamental
form of $x$, respectively. The eigenvalues
of $(B_{ij})$ are called the M\"{o}bius principal curvatures of $x$. The
covariant derivative of $C_i, A_{ij}, B_{ij}$ are defined by
\begin{eqnarray*}
&&\sum_jC_{i,j}\omega_j=dC_i+\sum_jC_j\omega_{ji},\label{co1}\\
&&\sum_kA_{ij,k}\omega_k=dA_{ij}+\sum_kA_{ik}\omega_{kj}+\sum_kA_{kj}\omega_{ki},\label{co2}\\
&&\sum_kB_{ij,k}\omega_k=dB_{ij}+\sum_kB_{ik}\omega_{kj}+\sum_kB_{kj}\omega_{ki}.\label{co3}
\end{eqnarray*}
The integrability conditions for the structure equations are given
by
\begin{eqnarray}
&&A_{ij,k}-A_{ik,j}=B_{ik}C_j-B_{ij}C_k,\label{equa1}\\
&&C_{i,j}-C_{j,i}=\sum_k(B_{ik}A_{kj}-B_{jk}A_{ki}),\label{equa2}\\
&&B_{ij,k}-B_{ik,j}=\delta_{ij}C_k-\delta_{ik}C_j,~~\sum_jB_{ij,j}=-(n-1)C_i,\label{equa3}\\
&&R_{ijkl}=B_{ik}B_{jl}-B_{il}B_{jk}+\delta_{ik}A_{jl}+\delta_{jl}A_{ik}-\delta_{il}A_{jk}-\delta_{jk}A_{il},\label{equa4}\\
&&\sum_iB_{ii}=0, ~~\sum_{ij}(B_{ij})^2=\frac{n-1}{n},~~
trA=\sum_iA_{ii}=\frac{1}{2n}(1+n^2s),\label{equa6}
\end{eqnarray}
where $R_{ijkl}$ denote the curvature tensor of $g$,
$s=\frac{1}{n(n-1)}\sum_{ij}R_{ijij}$ is the normalized
M\"{o}bius scalar curvature. When $n\geq 3$, we have the following fundamental theorem of hypersurfaces in M\"{o}bius geometry.
\begin{Theorem}(\cite{w})
Two hypersurfaces $x: M^n\mapsto \mathbb{S}^{n+1}$ and $\tilde{x}:
M^n\mapsto \mathbb{S}^{n+1} (n\geq 3)$ are M\"{o}bius equivalent if and only
if there exists a diffeomorphism $\varphi: M^n\rightarrow M^n$,
which preserves the M\"{o}bius metric $g$ and the M\"{o}bius second
fundamental form $B$.
\end{Theorem}

The coefficients of the M\"{o}bius second fundamental form and the M\"{o}bius form can be calculated in terms of the geometry of $x$ in $\mathbb{S}^{n+1}$ or $\sigma\circ f$ in $\mathbb{R}^{n+1}$ (see \cite{lim3,w}), under a local orthonormal frame $\{e_1,\cdots,e_n\}$,
\begin{equation}\label{cre2}
\begin{split}
B_{ij}&=\rho^{-1}(h_{ij}-H\delta_{ij}),\\
C_i&=-\rho^{-2}[e_i(H)+\sum_j(h_{ij}-H\delta_{ij})e_j(\log\rho)].
\end{split}
\end{equation}
Let $\{b_1,\cdots,b_n\}$ be the M\"{o}bius principal curvatures and
$\{\lambda_1,\cdots,\lambda_n\}$ be the principal curvatures of $x$,
then, from \eqref{cre2},
\begin{equation}\label{prin-cur}
b_i=\rho^{-1}(\lambda_i-H).
\end{equation}
Clearly the number of distinct M\"{o}bius principal curvatures
is the same as that of its distinct principal curvatures.

Next  we  calculate the M\"{o}bius invariants of the cone in $\mathbb{R}^{n+1}$,  and use these  M\"{o}bius invariants to characterize the cone in $\mathbb{R}^{n+1}$.

For $1 \leq m \leq n-1$,
let $u: {M}^m\longrightarrow \mathbb{S}^{m+1}\subset \mathbb{R}^{m+2}$ be an immersed hypersurface in
$\mathbb{S}^{m+1}$. The cone over $u$ in $\mathbb{R}^{n+1}$ is given as
\begin{equation*}
f(t,y,p)=(y,tu(p)): \mathbb{R}^+\times \mathbb{R}^{n-m-1}\times {M}^m\longrightarrow \mathbb{R}^{n+1}.
\end{equation*}

It is easily calculated that the first and the second fundamental form of the cone $f$ is
\[
I_f =  dt^2 + |dy|^2 +t^2 I_u, ~~~II_f = t~II_u,\]
where $I_u$ and $II_u$ are the first and
second fundamental forms of the hypersurface $u$ in the sphere $\mathbb{S}^{m+1}$
respectively. Let $H_u$ denote the mean curvature of $u$. The principal curvatures of the cone $f$ are
\begin{equation}\label{princ}
\underbrace{0,\cdots,0}_{n-m},\frac{1}{t}\lambda_1,\cdots,\frac{1}{t}\lambda_m,
\end{equation}
where $\{\lambda_1,\cdots,\lambda_m\}$ are the principal curvatures of $u$. Hence
$$
\rho^2 = \frac  n{n-1}\left(|II_u|^2 - \frac {m^2}nH_u^2\right) \frac 1{t^2}.
$$
The M\"{o}bius metric of the cone is
\begin{equation}\label{metric}
g = \rho^2 I_f =  \rho_0^2 (\frac {dt^2 + |dy|^2}{t^2} + I_u),
\end{equation}
and the M\"{o}bius position vector of the cone  is
$$
Y(t, y, p) =\frac {\rho_0}t (\frac{1+t^2+|y|^2}{2},\frac{1-t^2-|y|^2}{2}, y, tu(p)):
\mathbb{R}^+\times \mathbb{R}^{n-m-1}\times{M}^m\to
C^{n+2}_+\subset \mathbb{R}^{n+3}_1,
$$
where $\rho_0^2=\frac{n}{n-1}(|II_u|^2-\frac{m^2}{n}H_u^2)$. Note that
\begin{equation*}
i(t, y) = (\frac{1+t^2+|y|^2}{2t},\frac{1-t^2-|y|^2}{2t},\frac{y}{t}):
\mathbb{R}^+\times \mathbb{R}^{n-m-1}=\mathbb{H}^{n-m}\to \mathbb{H}^{n-m}\subset \mathbb{R}^{n-m+1}_1
\end{equation*}
is nothing but the identity map of $\mathbb{H}^{n-m}$, since $\mathbb{R}^+\times \mathbb{R}^{n-m-1}=\mathbb{H}^{n-m}$ is the upper
half-space endowed with the standard hyperbolic metric. We may now rewrite the M\"{o}bius position vector of the cone $f$ as
\begin{equation}\label{mop2}
Y=\rho_0(i(t, y), u):\mathbb{R}^+\times \mathbb{R}^{n-m-1}\times{M}^m \to \mathbb{C}^{n+2}_+\subset \mathbb{R}^{n+3}_1.
\end{equation}
Consequently we have

\begin{Lemma}\label{redu}
Let $u:{M}^m\to \mathbb{S}^{m+1}$ be an immersed hypersurface in $\mathbb{S}^{m+1}\subset\mathbb{R}^{m+2}$ and
 \begin{equation*}
\frac 1{\rho_0}Y= (i(t, y), u):\mathbb{R}^+\times \mathbb{R}^{n-m-1}\times{M}^m\to \mathbb{H}^{n-m}\times\mathbb{S}^{m+1}\subset \mathbb{R}^{n+3}_1
\end{equation*}
for smooth positive function $\rho_0$. Suppose that $Y$ is the M\"{o}bius position vector for an immersed hypersurafce
$$f:\mathbb{R}^+\times \mathbb{R}^{n-m-1}\times{M}^m\to \mathbb{R}^{n+1}.$$
Then $f$ is a cone over $u$ and $\rho_0^2=\frac{n}{n-1}(|II_u|^2-\frac{m^2}{n}H_u^2)$.
\end{Lemma}
Using (\ref{prin-cur}), Let $\mu=\frac{n-1}{n\sqrt{|II_u|^2 - \frac {m^2}nH_u^2}}$, we can obtain the M\"{o}bius principal curvatures of the cone as follows.
$$\underbrace{-\frac{m}{n}\mu H_u,\cdots,-\frac{m}{n}\mu H_u}_{n-m},
\mu(\lambda_1-\frac{m}{n}H_u),\cdots,\mu(\lambda_m-\frac{m}{n}H_u).$$
Using (\ref{cre2}), we have the following results,
\begin{Corollary}\label{cone1}
Let $u:M^m\to \mathbb{S}^{m+1}$ be a homogeneous hypersurface, then the M\"{o}bius form of the cone hypersurface in $\mathbb{R}^{n+1}$ over $u$ vanishes, i.e., $C=0$.
\end{Corollary}

\section{The proof of main theorem \ref{th11}}
Let $x: M^n\mapsto \mathbb{S}^{n+1} (n\geq 2)$ be a M\"{o}bius homogeneous hypersurface.
If there exists an umbilical point in $M^n$,
then the hypersurface $x$ is a totally umbilical hypersurface which can be viewed as an isometric homogeneous hypersurface in $\mathbb{S}^{n+1}$. So this is a trivial case.

In the rest we assume that the M\"{o}bius homogeneous hypersurface $x$ is umbilic-free. The M\"{o}bius invariants $g, B, A$ and $C$ can be defined on the hypersurface.
The proof of the Main Theorem \ref{th11} depends on the following key result.
\begin{Theorem}\label{the1}
Let $x: M^n\to \mathbb{S}^{n+1}$ be a M\"{o}bius homogeneous hypersurface with $r$ distinct principal curvatures. If  $r>2$,
then the M\"{o}bius form vanishing, i.e., $C=0$.
\end{Theorem}
The proof of the Vanishing Theorem \ref{the1} is long and we postpone this to next section.
Now we use it to prove the classification Theorem \ref{th11}.

The case $r=2$ was already classified \cite{lim}.
\begin{Theorem}(\cite{lim})\label{the4-1}
Let $x: M^n\mapsto \mathbb{S}^{n+1}$ be a M\"{o}bius homogeneous hypersurface with two distinct principal curvatures.
Then $x$ is M\"{o}bius equivalent to one of
the following hypersurfaces:\\
(1) the standard torus $\mathbb{S}^k(r)\times \mathbb{S}^{n-k}(\sqrt{1-r^2})$, $1\leq k\leq n-1$;\\
(2) the images of $\sigma$ of the standard cylinder $\mathbb{S}^k(1)\times \mathbb{R}^{n-k}\subset \mathbb{R}^{n+1}$, $1\leq k\leq n-1$;\\
(3) the images of $\tau$ of  $\mathbb{S}^k(r)\times \mathbb{H}^{n-k}(\sqrt{1+r^2})$, $1\leq k\leq n-1$;\\
(4) the image of $\sigma$ of a logarithmic spiral cylinder.
\end{Theorem}

If the M\"{o}bius homogeneous hypersurface $x$ has $r(>2)$ distinct principal curvatures, then $C=0$ by the Vanishing Theorem \ref{the1}.
Thus the M\"{o}bius homogeneous hypersurface $x$ is a M\"{o}bius isoparametric hypersurface. A hypersurface in $\mathbb{S}^{n+1}$ is called a M\"{o}bius isoparametric hypersurface if its M\"{o}bius principal curvatures $\{b_1,\cdots,b_n\}$ are constant and the M\"{o}bius form vanishes.
In \cite{lim3}, the first and third author (collaborated with Jie Qing) proved that a hypersurface is Dupin hypersurface with constant M\"{o}bius curvatures if and only if it is a M\"{o}bius isoparametric hypersurface.
The M\"{o}bius curvatures $M_{ijk}$ of hypersurface are defined as the ratio of its principal curvatures (M\"{o}bius principal curvatures),
$$
M_{ijk}=\frac{\lambda_i-\lambda_j}{\lambda_i-\lambda_k}=\frac{b_i-b_j}{b_i-b_k}.
$$
Also in \cite{lim3}, we have classified the Dupin hypersurfaces with constant M\"{o}bius curvatures, that is, the M\"{o}bius isoparametric hypersurfaces were classified.
\begin{Theorem}\cite{lim3}\label{the4-2}
Let $M^n$ be a Dupin hypersurface in $\mathbb{S}^{n+1}$ with $r (\geq 3)$ distinct principal curvatures. If the  M\"{o}bius curvatures are constant, then locally $M^n$
is M\"{o}bius equivalent to one of the following hypersurfaces: \\
(1) an isoparametric hypersurface in $\mathbb{S}^{n+1}$;\\
(2) the image of the stereograph projection of a cone over an isoparametric hypersurface in $\mathbb{S}^k\subset \mathbb{R}^{k+1}\subset\mathbb{R}^{n+1}$.
\end{Theorem}
Using these results, the M\"{o}bius homogeneous hypersurfaces with $r>2$ distinct principal curvatures are M\"{o}bius equivalent to
one of the following hypersurfaces: \\
(1) an isometric homogeneous hypersurface in $\mathbb{S}^{n+1}$;\\
(2) the image of the stereograph projection of a cone hypersurface over an isometric homogeneous hypersurface in $\mathbb{S}^k\subset \mathbb{R}^{k+1}\subset\mathbb{R}^{n+1}$.

Theorem \ref{the4-1} and Theorem \ref{the4-2} imply Theorem \ref{th11} and the classification immediately.

The remaining part of this paper is devoted to the proof of the key result, the Vanishing Theorem~\ref{the1}.

\section{The orbit geometry and three cases}

To prove the Vanishing Theorem \ref{the1}, we need some facts and properties about subgroups of $\textup{O}^+(n+2,1)$, which we refer to \cite{olmos}. This geometrical viewpoints will help us see clearly the underlying structure of these homogeneous hypersurfaces.

We consider the action of $\textup{O}^+(n+2,1)$ on the hyperbolic space,
which is nothing else but the isometry group of $\mathbb{H}^{n+2}$.
Let $\partial_{\infty}\mathbb{H}^{n+2}$ denote the boundary at infinity of the hyperbolic space $\mathbb{H}^{n+2}$, which is diffeomorphic to $PC^{n+1}\approx \mathbb{S}^{n+1}$.
A horosphere in $\mathbb{H}^{n+2}$ is the intersection of $\mathbb{H}^{n+2}$ with a degenerate hyperplane. Any horosphere is completely determined
by specifying a point in $\mathbb{H}^{n+2}$ and another point in $\partial_{\infty}\mathbb{H}^{n+2}$. It is a well known fact that horospheres are umbilical
flat submanifolds of codimension one which are isometric to a Euclidean space.

An action of a subgroup $G$ of $\textup{O}^+(n+2,1)$ on $\mathbb{R}^{n+3}_1$ is called
\emph{weakly irreducible}
if it leaves invariant degenerate subspaces only.
An action of a subgroup $G$ on a manifold $M$ is
\emph{transitive}
if $M$ is an orbit of $G$, i.e., $M=G\cdot p,~p\in M$.
\begin{Theorem}\cite{olmos}\label{4-20}
Let $G$ be a connected Lie subgroup of $\textup{O}^+(n+2,1)$ and assume that the action of $G$ on $\mathbb{R}^{n+3}_1$ is irreducible,
then $G=\textup{O}^+(n+2,1)$.
\end{Theorem}
\begin{Theorem}\cite{olmos}\label{4-2}
Let $G$ be a connected Lie subgroup of $\textup{O}^+(n+2,1)$ and assume that the action of $G$ on  $\mathbb{R}^{n+3}_1$
is weakly irreducible. Then $G$ acts transitively either on $\mathbb{H}^{n+2}$ or on a horosphere of $\mathbb{H}^{n+2}$. Moreover, if $G$
acts irreducibly, then $G=\textup{O}^+(n+2,1)$.
\end{Theorem}

Let $M^k$ a submanifold  of $\mathbb{H}^{n+2}$, and $\nu(M^k)$ be its normal bundle. Let $exp: \nu(M^k)\to \mathbb{H}^{n+2}$ be the normal exponential map.
If $\xi$ is a parallel unit normal vector field on $M^k$, then one can construct a new set $M^k_t$  parallel to $M^k$ as follows,
$$
M^k_t=exp(t\xi(p)),~~ \forall p\in M^k.
$$
Now let $\mathbb{S}_q$ be a horosphere centered at $q\in \partial_{\infty}\mathbb{H}^{n+2}$ and
$\xi$ the unit normal vector field, then the foliation by
horospheres centered at $q$ coincides with the foliation $\mathbb{S}_t, t\in R$. It is a well known fact that the foliation by horospheres centered
$q\in \partial_{\infty}\mathbb{H}^{n+2}$ coincides with the foliation given by the intersection of $\mathbb{H}^{n+2}$ with a family of parallel degenerate
hyperplanes of $\mathbb{R}^{n+3}_1$.

\begin{lemma}\label{horo1}\cite{olmos}
If $M^k=G\cdot p,~~ p\in \mathbb{H}^{n+2}$, where $G$ acts by isometries on $\mathbb{H}^{n+2}$ and $\xi$ is an equivariant unit normal vector field over $M^k$,
we have that $M^{n+1}_t=G\cdot exp(t\xi(p))$ and therefore it is also a $G$-orbit.
If $exp(t\xi)$ is an immersion, then $M^k_t$ is the parallel manifold to $M^k$.
\end{lemma}
The following Theorem is key in describing subgroup behavior of $\textup{O}^+(n+2,1)$.
\begin{Theorem}\cite{olmos}\label{4-1}
Let $G$ be a connected Lie subgroup of $\textup{O}^+(n+2,1)$. If we consider the action of the subgroup  $G$ on $\mathbb{H}^{n+2}$. Then one of the following assertions holds true:\\
(i)  $G$ has a fixed point in $\mathbb{H}^{n+2}$.\\
(ii)  $G$ has a unique non-trivial totally geodesic orbit (possibly $\mathbb{H}^{n+2}$).\\
(iii) All orbits are contained in horospheres centered at the same point at $\partial_{\infty}\mathbb{H}^{n+2}$.
\end{Theorem}

By Theorem \ref{4-1} we divide the proof of the Vanishing Theorem \ref{the1} into three cases, which correspond to the following three propositions.

\begin{PROPOSITION}\label{prop3-1}
Let $x: M^n\mapsto \mathbb{S}^{n+1} (n\geq 2)$ be a M\"{o}bius homogeneous hypersurface with $x(M^n)=\Psi(G)\cdot p$
for some subgroup $G\subset \textup{O}^+(n+2,1)$.
If the subgroup $G$ has a fixed point $q\in \mathbb{H}^{n+2}$, then the M\"{o}bius form of $x$ vanishes, i.e., $C=0$.
\end{PROPOSITION}

\begin{PROPOSITION}\label{prop3-3}
Let $x: M^n\mapsto \mathbb{S}^{n+1} (n\geq 2)$ be a M\"{o}bius homogeneous hypersurface with $x(M^n)=\Psi(G)\cdot p$
for some subgroup $G\subset \textup{O}^+(n+2,1)$. If
the subgroup $G$ has a unique non-trivial totally geodesic orbit (possibly $\mathbb{H}^{n+2}$), then $C=0$.
\end{PROPOSITION}

\begin{PROPOSITION}\label{prop3-2}
Let $x: M^n\mapsto \mathbb{S}^{n+1} (n\geq 2)$ be a M\"{o}bius homogeneous hypersurface with $x(M^n)=\Psi(G)\cdot p$
for some subgroup $G\subset \textup{O}^+(n+2,1)$.
Suppose all orbits of $G$ in $\mathbb{H}^{n+2}$ are contained in horospheres centered at the same point $z\in \partial_{\infty}\mathbb{H}^{n+2}$. Moreover, assume the number of the distinct
 principal curvatures $r\geq 3$. Then $C=0$.
\end{PROPOSITION}
The Vanishing Theorem \ref{the1} follows if we can prove these three propositions.

\subsection{{\bf Case 1}: $G$ has a fixed point in $\mathbb{H}^{n+2}$}
In this subsection, we prove Proposition \ref{prop3-1}.

Let $q=(y_0,y_1,\cdots, y_{n+1})\in \mathbb{H}^{n+2}$ be a fixed point, we can assume that $y_0> 0$, otherwise we take $-q$ instead of $q$.
There exists $T\in \textup{O}^+(n+2,1)$ such that
$$T(q)=(a,0,\cdots,0)^t,~~a>0.$$
Since $q$ is a fixed point, for all $h\in G$,
$$T\circ h\circ T^{-1}\cdot (a,0,\cdots,0)^t=Th\cdot q=T\cdot q=(a,0,\cdots,0)^t.$$
Setting $T\circ h\circ T^{-1}=\left(\begin{matrix}w&u\\
v&Q\end{matrix}\right),$ then
$$\left(\begin{matrix}w&u\\
v&Q\end{matrix}\right)\left(\begin{matrix}a\\0\\\vdots\\
0\end{matrix}\right)=\left(\begin{matrix}a\\0\\\vdots\\
0\end{matrix}\right),$$
which implies that $w=1, u=v=0$ and $Q\in \textup{O}(n+2)$, that is $TGT^{-1}\subset \Psi^{-1}(\textup{O}(n+2))$.
Thus the M\"{o}bius homogeneous hypersurface $M^n$ is an orbit of the isometry group $\textup{O}(n+2)$ of $\mathbb{S}^{n+1}$ up to a M\"{o}bius transformation.
So up to a  M\"{o}bius transformation the hypersurface $x(M^n)$ is an isometric homogeneous hypersurface, whose mean curvature and  principal curvatures are constant. We know that $C=0$ by equation (\ref{cre2}). Thus Proposition \ref{prop3-1} is proved.

\subsection{{\bf Case 2}:  $G$ has a unique non trivial totally geodesic orbit in $\mathbb{H}^{n+2}$}

To prove Proposition \ref{prop3-3}, we need the following lemmas.
\begin{lemma}\label{lemm6-1}
Let $x: M^n\mapsto \mathbb{S}^{n+1} (n\geq 2)$ be a M\"{o}bius homogeneous hypersurface with $x(M^n)=\Psi(G)\cdot p$.
If the group $G$ has a unique non-trivial totally geodesic orbit $\mathbb{H}^{n-m}$ in $\mathbb{H}^{n+2}$, then there exists a subgroup $G_1$
 of $\textup{O}(m+2)$ and $G$ is the direct product of $\textup{O}^+(n-m,1)$ and $G_1$, i.e., $$G=\textup{O}^+(n-m,1)\oplus G_1,~~1\leq m\leq n-1.$$
\end{lemma}
\begin{proof}
 Since $\mathbb{H}^{n-m} (1\leq m\leq n-1)$ is the unique non-trivial totally geodesic orbit of $G$ in $\mathbb{H}^{n+2}$, there exist a non-trivial subspace $V$ in $\mathbb{R}^{n+3}_1$ such that $\mathbb{H}^{n-m}\subset V$. Let the non-trivial subspace $V=\mathbb{R}^{n-m+1}_1$ up to a transformation of $\textup{O}^+(n+2,1)$,
thus
$$\mathbb{H}^{n-m}=\mathbb{R}^{n-m+1}_1\cap\mathbb{H}^{n+2}.$$
The orthogonal complement space $(\mathbb{R}^{n-m+1}_1)^{\bot}=\mathbb{R}^{m+2}$ of $\mathbb{R}^{n-m+1}_1$ is also an invariant subspace of $G$.
Thus there exist two subgroups $G_0, G_1$
such that $G$ is the direct product of $G_0$ and $G_1$, i.e., $G=G_0\oplus G_1$.
$G_1$ acts trivially on $\mathbb{R}^{n-m+1}_1$
and $G_0$ acts trivially on $(\mathbb{R}^{n-m+1}_1)^{\bot}$.
Since the non-trivial totally geodesic orbit is unique, then $G_0$ acts irreducibly
on $\mathbb{R}^{n-m+1}_1$.
By Theorem \ref{4-20} we have $G_0=\textup{O}^+(n-m,1)$. Since $\mathbb{R}^{n+3}_1=\mathbb{R}^{n-m+1}_1\oplus\mathbb{R}^{m+2}$, the invariant subspace $(\mathbb{R}^{n-m+1}_1)^{\bot}=\mathbb{R}^{m+2}$ is positive definite, then $G_1\subset Iso(\mathbb{R}^{m+2})$.
Since $G_1$ acts linearly on $\mathbb{R}^{m+2}$, then $G_1\subset\textup{O}(m+2)$.
Thus $G_1$ is a subgroup of $\textup{O}(m+2)$ and $G=\textup{O}^+(n-m,1)\oplus G_1.$
\end{proof}

\begin{lemma}\label{pro-3}
Let $x: M^n\mapsto \mathbb{S}^{n+1} (n\geq 2)$ be a M\"{o}bius homogeneous hypersurface with $x(M^n)=\Psi(G)\cdot p$.
Suppose the group $G=\textup{O}^+(n-m,1)\oplus G_1,~~1\leq m\leq n-1,$ where $G_1$ is a subgroup
 of $\textup{O}(m+2)$. Then $x$ is a cone hypersurface over a homogeneous hypersurface $u:M^m\to \mathbb{S}^{m+1}$.
\end{lemma}
\begin{proof}
Since $\mathbb{R}^{n+3}_1=\mathbb{R}^{n-m+1}_1\oplus\mathbb{R}^{m+2}$, for $p\in\mathbb{C}^{n+2}_+$, $p=(\bar{p}_1,\bar{p}_2)$,
where $\bar{p}_1\in \mathbb{R}^{n-m+1}_1$ and $\bar{p}_2\in \mathbb{R}^{m+2}$. Since $\langle p,p\rangle=0$, up to a scaling, we can assume that
$$\langle \bar{p}_1,\bar{p}_1\rangle=-1,~~\langle \bar{p}_2,\bar{p}_1\rangle=1,$$
thus $\bar{p}_1\in \mathbb{H}^{n-m}\subset\mathbb{R}^{n-m+1}_1$ and $\bar{p}_2\in \mathbb{S}^{m+1}\subset\mathbb{R}^{m+2}$.
$$Y(M^n)=G\cdot p=\textup{O}^+(n-m,1)\oplus G_1\cdot (\bar{p}_1,\bar{p}_2)=\mathbb{H}^{n-m}\times G_1\cdot \bar{p}_2.$$
Since $G_1\subset \textup{O}(m+2)$, thus $G_1\cdot \bar{p}_2$ is a homogeneous hypersurface in $\mathbb{S}^{m+1}$.
Thus  $x$ is a cone hypersurface over a homogeneous hypersurface $u:M^m\to \mathbb{S}^{m+1}$ by Lemma \ref{redu}.
\end{proof}

Recalling Corollary \ref{cone1}, Proposition \ref{prop3-3} is proved.

\subsection{{\bf Case 3}: All orbits in $\mathbb{H}^{n+2}$ contained in horospheres centered at the same point at $\partial_{\infty}\mathbb{H}^{n+2}$}

In order to prove Proposition~\ref{prop3-2}, we need some preparation.

Since $Y:(M^n,g)\to \mathbb{C}^{n+2}_+\subset\mathbb{R}^{n+3}_1$ is an isometric immersion,
it is convenient to consider the M\"{o}bius homogeneous hypersurface $Y(M^n)$.
A vector field $X\in TY(M^n)\subset\mathbb{R}^{n+3}_1$ is called
\emph{$G$-equivariant}
if
$$h_*(X_p)=X_{h(p)},~~~ for ~~p\in M^n~~ and ~~h\in G.$$
If $X,Y$ are two $G$-equivariant vector fields on $M^n$, then $\nabla_XY$ also is $G$-equivariant for $G$ is a subgroup of the isometric group of $M^n$. Obviously, the simple principal vector fields  $\{E_1,E_2,\cdots,E_s\}$ are $G$-equivariant.

\begin{lemma}\label{case2-1}
Let $x: M^n\mapsto \mathbb{S}^{n+1} (n\geq 2)$ be a M\"{o}bius homogeneous hypersurface with $x(M^n)=\Psi(G)\cdot p$. If the vector field $X$ on $Y(M^n)$ is $G$-equivariant, then the connected integral curves are homogeneous.
\end{lemma}
\begin{proof}
Let $p\in Y(M^n)$ and $\gamma$ is the connected integral curve of $X$ through $p$. For any point $q\in \gamma$ and $p\neq q$, there exists an element $h\in G$ such that $h(p)=q$ because of the homogeneity of $Y(M^n)$. Since $X$ is $G$-equivariant, so $h_*(X_p)=X_{h(p)}=X_q$, and $h\cdot \gamma$ is a connected integral curve of $X$ through $hp$. Thus $h\cdot \gamma=\gamma$.
Define the subgroup $$G_1=\{h\in G|hp\in \gamma\}.$$
Therefore the connected integral curve $\gamma$ is homogeneous, which is a orbit of $G_1$.
\end{proof}

Let $\mathfrak{g}$ be the Lie algebra of $G$. For each $X^*$ in $\mathfrak{g}$ we denote $Exp(tX^*)$
the one-parameter subgroup of $G$ generated by $X^*$. The action of $Exp(tX^*)$ on $M^n$ turns it into a one-parameter group $\varphi_t$ of diffeomorphisms
of $M^n$, defined by
$$\varphi_t(p)=Exp(tX^*)p.$$
We will identify $X^*$ in $\mathfrak{g}$ with the vector field $X$ on $M^n$ generated by $\varphi_t$, i.e.,
$$X_p=\frac{d}{dt}|_{t=0}\varphi_t(p),~~p\in M^n.$$

Let $\Gamma_p$ be the isotropy subgroup at the point $p\in M$, that is
$$\Gamma_p=\{\varphi\in G|\varphi(p)=p\}.$$
 Let $\mathfrak{k}$ the Lie subalgebra of $\mathfrak{g}$ generated by $\Gamma_p$, then
there exists a subspace $\mathfrak{p}$ of $\mathfrak{g}$ such that $\mathfrak{g}$ is the direct sum
$$\mathfrak{g}=\mathfrak{k}\oplus\mathfrak{p}.$$
We can identify $\mathfrak{p}$ with $T_pM$. Thus the following result is obvious.
\begin{lemma}\label{case2-0}
Let $(M^n,g)$ be a $G$-homogeneous Riemannian manifold, and $\gamma$ is a $G$-homogeneous curve in $M^n$,
then there exists an one-parameter subgroup $G_1$ of $G$ such that the curve $\gamma$ is the orbit of $G_1$
\end{lemma}
\begin{lemma}\label{case2-11}
Let $x: M^n\mapsto \mathbb{S}^{n+1} (n\geq 2)$ be a M\"{o}bius homogeneous hypersurface with $x(M^n)=\Psi(G)\cdot p$,  and
all orbits of $G$ in $\mathbb{H}^{n+2}$ are contained in horospheres centered at the same point  $z\in \partial_{\infty}\mathbb{H}^{n+2}$. Let $X$ be an $G$-equivariant vector field on $Y(M^n)$ and $\gamma$ its integral curves, then $\gamma$ is a plane curve in $\mathbb{R}^{n+3}_1$.
\end{lemma}
\begin{proof}
Since $X$ is $G$-equivariant, then the integral curves are homogeneous. Now Let $\gamma$ be an integral curve of $X$ passing $p_0\in Y(M^n)$, then there exists an one-parameter subgroup $G_1$ of $G$ such that the integral curve $\gamma$ is the orbit of $G_1$, i,e, $\gamma=G_1\cdot p_0$.

Since $G_1$ is an one parameter subgroup, then for almost all $\bar{p}\in \mathbb{H}^{n+2}$, the orbit $G_1\cdot \bar{p}$ is a curve $\tilde{\gamma}$ in $\mathbb{H}^{n+2}$.
Now let $V$ be an invariant subspace and $z\in V$, and we assume $G_1$ acts irreducibly (or weakly irreducibly) on $V$, then $dim V=3$.
Otherwise,
if $dim V=l>3$, by Theorem \ref{4-2} and Theorem \ref{4-20}, either $G_1$ acts transitively on $\mathbb{H}^{l-1}\subset V$ or $G_1$ acts transitively on a horosphere of $\mathbb{H}^{l-1}\subset V$, which is a contradiction because the orbit $G_1\cdot \bar{p}$ is a curve in $\mathbb{H}^{l-1}$.
Since all orbits of $G$ in $\mathbb{H}^{n+2}$ are contained in horospheres centered at the same point  $z\in \partial_{\infty}\mathbb{H}^{n+2}$, thus all orbits of $G_1$ in $\mathbb{\bar{H}}^2=V\cap \mathbb{H}^{n+2}$ are included in horocycles of $\mathbb{\bar{H}}^2$ centered at $z$. The action of $G_1$ on $V$ is weakly irreducible, thus $G$ acts transitively  on a horocycle of $\mathbb{H}^{2}$. So the orbit $\tilde{\gamma}=G_1\cdot \bar{p}$ is a horocycle in $\mathbb{H}^{n+2}$.

Since $\tilde{\gamma}$ is a horocycle, its normal bundle in $\mathbb{H}^{n+2}$ is trivial, then there exists a parallel normal vector field $\xi$ of $\tilde{\gamma}$ in $\mathbb{H}^{n+2}$.
We define a parallel curve $\gamma_{\xi}$ of $\tilde{\gamma}$ in $\mathbb{H}^{n+2}$, then
$$\gamma_{\xi}=\{exp_{\bar{q}}(\xi_{\bar{q}})|\bar{q}\in G_1\cdot \bar{p}\}=G_1\cdot exp_{\bar{p}}(\xi_{\bar{p}})$$
 is the orbit through $exp_{\bar{p}}(\xi_{\bar{p}})$ of $G_1$. Thus each orbit of $G_1$ in $\mathbb{H}^{n+2}$ can be obtained in this manner from a single principal orbit $\tilde{\gamma}$ by Lemma \ref{horo1}.

Since each horocycle in $\mathbb{\bar{H}}^2$ is the intersection of $\mathbb{\bar{H}}^2$ with a degenerate plane $\prod$, so for all $\bar{p}\in \mathbb{H}^{n+2}$, the orbit $G_1\cdot \bar{p}$ is a  horocycle.

Let $\mathbb{H}^{n+2}_-=\{-\bar{p}|\bar{p}\in\mathbb{H}^{n+2}\}$. Since The action of $G_1$ on $\mathbb{R}^{n+3}_1$ is linear, thus for all $\bar{p}\in \mathbb{H}^{n+2}_-$, the orbit $G_1\cdot \bar{p}$ is a horocycle.

Since for a point $\bar{p}\in \mathbb{H}^{n+2}$ or $\bar{p}\in \mathbb{H}^{n+2}_-$, the orbit $G_1\cdot \bar{p}$ is a  horocycle, which is the intersection of a degenerate plane $\prod(\bar{p})$ with $\mathbb{H}^{n+2}$ or $\mathbb{H}^{n+2}_-$. Let $\prod(z)$ denote the degenerate hypersurface $\{w|\langle w,z\rangle=0\},$ which satisfies $\mathbb{C}^{n+2}_+\bigcap\prod(z)=\{\lambda z|\lambda\in R\}$.
Therefore
$$\mathbb{R}^{n+3}_1=\Big(\bigcup_{\bar{p}\in \mathbb{H}^{n+2}}\prod(\bar{p})\Big)\bigcup\Big(\bigcup_{\bar{p}\in \mathbb{H}^{n+2}_-}\prod(\bar{p})\Big)\bigcup \prod(z).$$

Since $G_1$ preserves the foliation by horocycles centered $z\in \partial_{\infty}\mathbb{H}^{n+2}$, the affine plane $\prod(\bar{p})$ is $G_1$-invariant. In fact,
let $\tilde{\gamma}_{\bar{p}}$ be the  horocycle in $\prod(\bar{p})$, then $\tilde{\gamma}_{\bar{p}}=G_1\cdot \bar{p}$, that is, $\tilde{\gamma}_{\bar{p}}$ is an orbit of $G_1$ in $\prod(\bar{p})$.
For any $q\in \prod(\bar{p})$, we can choose three
points $\bar{p}_1,\bar{p}_2,\bar{p}_3\in \tilde{\gamma}_{\bar{p}}\subset\prod(\bar{p})$ such that
$$q=k_1\bar{p}_1+k_2\bar{p}_2+k_3\bar{p}_3,~k_1+k_2+k_3=1.$$
Since the action of $G$ on $\mathbb{R}^{n+3}_1$ is linear, thus
$$G_1\cdot q=k_1G_1\cdot \bar{p}_1+k_2G_1\cdot \bar{p}_2+k_3G_1\cdot \bar{p}_3,~~k_1+k_2+k_3=1.$$
Since $G_1\cdot \bar{p}_1, G_1\cdot \bar{p}_2, G_1\cdot \bar{p}_3\in \tilde{\gamma}_{\bar{p}}\subset \prod(\bar{p})$, thus
$G_1\cdot q\in \prod(\bar{p})$ and $$G_1\Big(\prod(\bar{p})\Big)\subset\prod(\bar{p}).$$

Now Let $\gamma$ be an integral curve of the equivariant vector field $X$ through $q\in Y(M^n)\subset \mathbb{C}^{n+2}_+$, then $\gamma=G_1\cdot q$.
There exists an affine plane $\prod(\bar{p})$ such that $q\in \prod(\bar{p})$.
Since $G_1\Big(\prod(\bar{p})\Big)\subset\prod(\bar{p})$, thus
 $\gamma\in \prod(\bar{p})$. Thus we finish the proof of Lemma \ref{case2-11}.
\end{proof}

The complete proof of Proposition~\ref{prop3-2} using the structure equations is left to the final section.

\section{ Proof of the Vanishing Theorem \ref{the1}}

In this section, we finish the proof of Proposition~\ref{prop3-2}. This needs to be done by combining the orbit geometry information and the structure equations. This will simplifies the study of the M\"{o}bius invariants.

Assume that the M\"{o}bius homogeneous hypersurface $x$ is umbilic-free. Then there exists a subgroup $G\subset\textup{O}^+(n+2,1)$ such that the orbit $G\cdot p$  is $Y(M^n)$, i.e.,
$$
G\cdot p=Y(M^n), ~~p\in Y(M^n).
$$
We can choose a local orthonormal frame $\{E_1,E_2,\cdots,E_n\}$ on $M^n$  such that
\begin{equation}\label{b31}
(B_{ij})=diag(b_1,\cdots,b_n)=diag(\bar{b}_1,\cdots,\bar{b}_s,\underbrace{\bar{b}_{s+1},\cdots,\bar{b}_{s+1}}_{m_{s+1}},
\cdots,\underbrace{\bar{b}_r,\cdots,\bar{b}_r}_{m_{r}}),
\end{equation}
where $\bar{b}_1,\cdots,\bar{b}_s$ are the distinct simple M\"{o}bius principal curvatures and $\bar{b}_{s+1},\cdots,\bar{b}_r$ are the distinct multiple M\"{o}bius principal curvatures.
Thus $$s+m_{s+1}+m_{s+2}+\cdots+m_r=n.$$ These M\"{o}bius principal curvatures are constant because of the homogeneity assumption.

 Let $\{\omega_1,\omega_2,\cdots,\omega_n\}$ be its dual basis with respect to
$\{E_1,E_2,\cdots,E_n\}$, and $\{\omega_{ij},~i,j=1,2,\cdots,n\}$ are the connection forms,
then $$\nabla_{E_j}^{E_i}=\sum_k\omega_{ik}(E_j)E_k.$$

 From
$\sum_kB_{ij,k}\omega_k=dB_{ij}+\sum_kB_{ik}\omega_{kj}+\sum_kB_{kj}\omega_{ki}$ and (\ref{b31}), we have
\begin{equation}\label{eqb11}
(b_i-b_j)\omega_{ij}=\sum_mB_{ij,m}\omega_m,~~~~B_{ii,m}=0.
\end{equation}
Under (\ref{b31}), for a fixed M\"{o}bius principal curvature $b_i$, let  $$[b_i]=\{j|b_j=\bar{b}_i\}.$$ It follows from (\ref{eqb11}) that
\begin{equation}\label{eqb1}
\begin{split}
&B_{ij,k}=0,~~~~~[b_i]=[b_j],~1\leq k\leq n,\\
&\omega_{ij}=\sum_k\frac{B_{ij,k}}{b_i-b_j},~~[b_i]\neq [b_j].
\end{split}
\end{equation}
Now we fix two simple principal vector fields $E_i, E_j$, by (\ref{eqb1}), we have
\begin{equation}\label{xxx-01}
\nabla_{E_i}^{E_i}=\sum_{m\neq i}\frac{C_m}{b_m-b_i}E_m,
~~~\nabla_{E_j}^{E_i}=\sum_{m\neq i}\frac{B_{im,j}}{b_i-b_m}E_m.
\end{equation}
For a multiple principal vector field $E_{\alpha}$ and a simple principal vector field $E_i$, by (\ref{eqb1}),
\begin{equation}\label{xxx-3}
\begin{split}
&\nabla_{E_{\alpha}}^{E_i}=\sum_{m\neq i}\frac{B_{im,\alpha}}{b_i-b_m}E_m,~~
\nabla_{E_{\alpha}}^{E_{\alpha}}=\sum_{m=1}^s\frac{C_m}{b_m-b_{\alpha}}E_m
+\sum_{m\in[b_{\alpha}],m\neq\alpha}\omega_{\alpha m}(E_{\alpha})E_m,\\
&\nabla_{E_i}^{E_{\alpha}}=\sum_{m=1}^s\frac{B_{\alpha m,i}}{b_{\alpha}-b_m}E_m+\sum_{m\in[b_{\alpha}],m\neq\alpha}\omega_{\alpha m}(E_i)E_m
+\sum_{m\geq s+1,m\not\in [b_{\alpha}]}\frac{B_{\alpha m,i}}{b_{\alpha}-b_m}E_m.
\end{split}
\end{equation}

\begin{lemma}\label{lemma1-31}
 Under the local orthonormal frame
$\{E_1,E_2,\cdots,E_n\}$ in (\ref{b31}), the coefficients of the M\"{o}bius form satisfy $$C_i=0,~~s+1 \leq i\leq n.$$
\end{lemma}
\begin{proof}
When $s+1 \leq i\leq n$, the principal curvatures  $b_i$ are multiple, and we can choose $j,k\in [b_i]$ such that $j\neq k.$ From the first equation of  (\ref{eqb1}), we have
$$B_{jj,k}=0,~~~~B_{jk,j}=0.$$
Combining with (\ref{equa3}),
$$0=B_{jj,k}-B_{jk,j}=C_k,~~~k\in [b_i].$$
Thus $C_j=0,~~j\in [b_i]$ and Lemma~\ref{lemma1-31} is proved.
\end{proof}

By Lemma \ref{lemma1-31}, the vanishing result in Theorem~\ref{the1} holds for these M\"{o}bius homogeneous hypersurfaces whose principal curvatures are all multiple.

\begin{Corollary}\label{coro-1}
If all principal curvatures of the M\"{o}bius homogeneous hypersurface are multiple, then the M\"{o}bius form vanishing, i.e., $C=0$.
\end{Corollary}

Now consider the more general case.
Since $\{b_1,\cdots,b_s\}$ are simple, its corresponding principal vector fields  $\{E_1,\cdots,E_s\}$ can be well defined,
and  the corresponding coefficients of the M\"{o}bius invariants are constant; for example, $\{C_i, A_{ij},A_{ij,k},B_{ij,k}, C_{i,j}, 1\leq i,j,k\leq s\}$ are constant.

From
$\sum_kC_{i,k}\omega_k=dC_{i}+\sum_kC_{k}\omega_{ki}$ and (\ref{eqb1}), we have
\begin{equation}\label{eqb2}
C_{i,j}=\sum_{k\notin [b_i]}\frac{C_kB_{ki,j}}{b_k-b_i}.
\end{equation}
Combining with (\ref{equa2}),
\begin{equation}\label{eqb2-1}
A_{ij}=\sum_{k\notin [b_i],[b_j]}\frac{C_kB_{ki,j}}{(b_k-b_i)(b_k-b_j)}+\frac{2C_iC_j}{(b_i-b_j)^2},~~\forall~~ [b_i]\neq [b_j].
\end{equation}
Using $\sum_mB_{ij,km}\omega_m=dB_{ij,k}+\sum_{m}B_{mj,k}\omega_{mi}++\sum_{m}B_{im,k}\omega_{mj}++\sum_{m}B_{ij,m}\omega_{mk}$ and (\ref{b31}),
we have
$$B_{ij,ij}=\sum_{m\notin [b_i],[b_j]}\frac{2B_{ij,m}^2}{b_m-b_i}+\sum_{m\notin [b_i],[b_j]}\frac{C_m^2}{b_m-b_j}+\frac{C_i^2}{b_j-b_i}~,$$
$$B_{ij,ji}=\sum_{m\notin [b_i],[b_j]}\frac{2B_{ij,m}^2}{b_m-b_j}+\sum_{m\notin [b_i],[b_j]}\frac{C_m^2}{b_m-b_i}+\frac{C_j^2}{b_i-b_j}~.$$
Combining with Ricci's identity $B_{ij,kl}-B_{ij,lk}=\sum_{m}B_{mj}R_{mikl}+\sum_{m}B_{mi}R_{mjkl}$, we have
\begin{equation}\label{eqb2-2}
R_{ijij}=\sum_{m\notin [b_i],[b_j]}\frac{2B_{ij,m}^2-C_m^2}{(b_m-b_i)(b_m-b_j)}-\frac{C_i^2+C_j^2}{(b_i-b_j)^2},~~\forall~ [b_i]\neq [b_j].
\end{equation}

Now let us combine these equations with the orbit geometry in case 3. Since the simple principal vector fields $\{E_1,E_2,\cdots,E_s\}$ are $G$-equivariant, then the integral curves of $E_i,~1\leq i\leq s$ are plane curves in $\mathbb{R}^{n+3}_1$ by Lemma~\ref{case2-11}. Let $\gamma$ be an integral curve of $E_i$ and $\bar{E}_i=E_i|_{\gamma}$,
then
$$\nabla_{\bar{E}_i}\nabla_{\bar{E}_i}^{\bar{E}_i}=\lambda \bar{E}_i,$$
where $\lambda$ is a constant. Thus $\langle\nabla_{\bar{E}_i}\nabla_{\bar{E}_i}^{\bar{E}_i},E_k\rangle=0$ for any $k\neq i$.
Since $E_i$ are $G$-equivariant, combining (\ref{eqb2-1}) and (\ref{xxx-01}) we have
\begin{equation}\label{xxx-12}
\sum_{m\neq i,k}\frac{C_mB_{mk,i}}{(b_m-b_i)(b_m-b_k)}=0,~~A_{ik}=\frac{2C_iC_k}{(b_i-b_k)^2},~~k\neq i.
\end{equation}

If  $s\geq 2$, for any two simple principal vector fields $E_i, E_j$ and some constant $x_1,x_2$,  the vector field $$X=x_1E_i+x_2E_j$$ also is $G$-equivariant.  Then the integral curves of $X$ are plane curves in $\mathbb{R}^{n+3}_1$, Thus we have
$\nabla_{\bar{X}}\nabla_{\bar{X}}^{\bar{X}}=\lambda {\bar{X}},$
where $\lambda$ is a constant and ${\bar{X}}=X|_{\gamma}$ for some integral curve.
Since $E_i, E_j,\nabla_{E_i}E_j,\nabla_{E_i}\nabla_{E_j}E_i$, etc, are $G$-equivariant, then inner product of any two of them, for example, $ \langle \nabla_{E_i}E_j,\nabla_{E_j}E_i\rangle$, is constant.  From $\langle X, x_2E_i-x_1E_j\rangle=0$, we  have the following equation
\begin{equation*}
\begin{split}
0&=x_2^4\langle\nabla_{E_j}\nabla_{E_j}^{E_j},E_i\rangle
-x_1^4\langle\nabla_{E_i}\nabla_{E_i}^{E_i},E_j\rangle\\
&+x_1^3x_2\{\langle\nabla_{E_i}\nabla_{E_i}^{E_i},E_i\rangle-\langle\nabla_{E_i}\nabla_{E_i}^{E_j}+\nabla_{E_i}\nabla_{E_j}^{E_i}+
\nabla_{E_j}\nabla_{E_i}^{E_i},E_j\rangle\}\\
&+x_1^2x_2^2\{\langle\nabla_{E_i}\nabla_{E_i}^{E_j}+\nabla_{E_i}\nabla_{E_j}^{E_i}+
\nabla_{E_j}\nabla_{E_i}^{E_i},E_i\rangle
-\langle\nabla_{E_i}\nabla_{E_j}^{E_j}+\nabla_{E_j}\nabla_{E_j}^{E_i}+
\nabla_{E_j}\nabla_{E_i}^{E_j},E_j\rangle\}\\
&+x_1x_2^3\{\langle\nabla_{E_i}\nabla_{E_j}^{E_j}+\nabla_{E_j}\nabla_{E_j}^{E_i}+
\nabla_{E_j}\nabla_{E_i}^{E_j},E_i\rangle-\langle\nabla_{E_j}\nabla_{E_j}^{E_j},E_j\rangle\}.
\end{split}
\end{equation*}
Since $x_1,x_2$ are arbitrary constants, by the above equation we have the following equations.
\begin{equation}\label{xxx-1}
\begin{split}
&\sum_{m\neq i,j}\frac{C_mB_{mj,i}}{(b_m-b_i)(b_m-b_j)}=0,\\
&\sum_{m\neq i,j}\frac{C_m^2+B_{mj,i}^2}{(b_m-b_i)(b_m-b_j)}+\sum_{m\neq j}\frac{B_{mj,i}^2}{(b_m-b_j)^2}-\sum_{m\neq i}\frac{C_m^2}{(b_m-b_i)^2}=0.
\end{split}
\end{equation}

For the fixed index $1\leq i,j\leq s$, if there exists some multiple principal curvature $b_{\alpha}, \alpha\geq s+1$ such that $\sum_{m\in [b_{\alpha}]}B_{ij,m}E_m\neq 0$, then
$$E_{\alpha}=\frac{\sum_{m\in [b_{\alpha}]}B_{ij,m}E_m}{|\sum_{m\in [b_{\alpha}]}B_{ij,m}E_m|}$$
is $G$-equivariant vector field.
We re-choose orthonormal basis in the eigen-space $V_{\alpha}=Span\{E_i|i\in[b_{\alpha}]\}$ with respect to the M\"{o}bius principal curvature
$b_{\alpha}$ such that under the new basis,
$$B_{12,\alpha}\neq 0,~~B_{12,\gamma}=0,~\gamma\in [b_{\alpha}],~\gamma\neq\alpha.$$
For any constant $x_1,x_2$, $X=x_1E_i+x_2E_{\alpha}$ is also $G$-equivariant vector fields, and we have
$\nabla_X\nabla_X^X=\lambda X$.
Let $\pi_1=\sum_{m\in[\alpha],m\neq\alpha}(\omega_{\alpha m}(E_i))^2$.  Combining equations (\ref{xxx-3}) and $\langle\nabla_X\nabla_X^X,x_2E_i-x_1E_{\alpha}\rangle=0$, we have the following equations:
\begin{equation}\label{xxx-2}
\begin{split}
&\sum_{m\neq i}\frac{C_mB_{m\alpha,i}}{(b_m-b_i)(b_m-b_{\alpha})}=0,\\
&\sum_{m\neq i,m\not\in[b_{\alpha}]}\frac{C_m^2+B_{m\alpha,i}^2}{(b_m-b_i)(b_m-b_{\alpha})}
+\sum_{m\not\in[b_{\alpha}]}\frac{B_{m\alpha,i}^2}{(b_m-b_{\alpha})^2}+\pi_1
-\sum_{m\neq i}\frac{C_m^2}{(b_m-b_i)^2}=0.
\end{split}
\end{equation}

\begin{lemma}\label{lemm4-1}
Under the basis (\ref{b31}), we assume that the simple principal curvatures satisfy
$$b_1<b_2<\cdots <b_s.$$  Then
$$ C_1C_2=C_3=\cdots=C_n=0.$$
\end{lemma}
\begin{proof}
If $s=1$, then $C_2=\cdots=C_n=0$ by Lemma~\ref{lemma1-31}, and Lemma~\ref{lemm4-1} is proved.

If $s\geq 2$, we consider the following two case:\\
{\bf Subcase 1}: there is no other multiple principal curvature $b_{\alpha}$ such that $b_1<b_{\alpha}< b_2$.\\
{\bf Subcase 2}: there are some multiple principal curvatures $b_{\alpha_1},\cdots,b_{\alpha_t}$ such that $$b_1<b_{\alpha_1}<\cdots<b_{\alpha_t}<b_2.$$

For {\bf Subcase 1}, let $i=1, j=2$ in the second  equation of (\ref{xxx-1}), we have
\begin{equation}\label{case01}
\sum_{m\neq 1,2}\frac{C_m^2(b_2-b_1)}{(b_m-b_1)^2(b_m-b_2)}+\sum_{m\neq 1,2}\frac{B_{2m,1}^2}{(b_m-b_1)(b_m-b_2)}
+\sum_{m\neq 1,2}\frac{B_{2m,1}^2}{(b_m-b_2)^2}=0.
\end{equation}
Since all items in (\ref{case01}) are nonnegative, there must be
\begin{equation}\label{case1-1}
C_m=B_{1m,2}=0,~~m\neq 1,2.
\end{equation}

For {\bf Subcase 2},
if there exists a multiple principal curvature $b_{\alpha_k}$  in $\{b_1<b_{\alpha_1}<\cdots<b_{\alpha_t}<b_2\}$ such that its principal vectors having $\sum_{m\in [b_{\alpha_k}]}B_{12,m}E_m\neq 0$, then
$$E_{\alpha_k}=\frac{\sum_{m\in [b_{\alpha_k}]}B_{12,m}E_m}{|\sum_{m\in [b_{\alpha_k}]}B_{12,m}E_m|}$$ is $G$-equivariant vector field.
We re-choose a basis $\{E_{\alpha_k},E_{\alpha_k+1},\cdots\}$
in the eigenvalue space $V_{\alpha}=Span\{E_i|i\in [b_{\alpha_k}]\}$ with respect to the principal curvature $b_{\alpha_k}$ such that under the basis
$$B_{12,\alpha_k}\neq 0,~~B_{12,\gamma}=0,~\gamma\in [b_{\alpha_k}], \gamma\neq \alpha_k.$$

If there exists a multiple principal curvature $b_{\alpha_j}$  in $\{b_1<b_{\alpha_1}<\cdots<b_{\alpha_t}<b_2\}$ such that $\sum_{m\in [b_{\alpha_j}]}B_{1\alpha_k,m}E_m\neq 0$ and $b_{\alpha_j}<b_{\alpha_k}$, then
$$E_{\alpha_j}=\frac{\sum_{m\in [b_{\alpha_j}]}B_{1\alpha_k,m}E_m}{|\sum_{m\in [b_{\alpha_j}]}B_{1\alpha_k,m}E_m|}$$ is $G$-equivariant vector field.
By induction, we can assume there exists a minimal multiple principal curvature $b_{\alpha_j}$  in $\{b_1<b_{\alpha_1}<\cdots<b_{\alpha_t}<b_2\}$ such that $E_{\alpha_j}$  is a $G$-equivariant vector field and
\begin{equation}\label{case2-10}
B_{1\alpha_j,m}=0,~~m\in[b_{\alpha}], ~~b_1<b_{\alpha}<b_{\alpha_j}.
\end{equation}
Let $i=1, \alpha=\alpha_j$ in the second equation in (\ref{xxx-2}) we have
\begin{equation}\label{case02-5}
\begin{split}
&\sum_{m\neq 1}\frac{C_m^2(b_{\alpha_j}-b_1)}{(b_m-b_1)^2(b_m-b_{\alpha_j})}+
\sum_{m\neq 1,m\not\in[b_{\alpha_j}]}\frac{B_{m\alpha_j,1}^2}{(b_m-b_1)(b_m-b_{\alpha_j})}\\
&+\sum_{m\not\in[b_{\alpha_j}]}\frac{B_{m\alpha_j,1}^2}{(b_m-b_{\alpha_j})^2}+\sum_{m\in[b_{\alpha_j}],m\neq\alpha_j}(\omega_{\alpha_i m}(E_1))^2=0.
\end{split}
\end{equation}
Since the all items in (\ref{case02-5}) are nonnegative, they must be zero. Thus
\begin{equation}\label{case2-01}
C_m=0,~~m\neq 1,
\end{equation}
which implies that $C_1C_2=C_3=\cdots=C_n=0.$

For {\bf Subcase 2}, if for any multiple principal curvature $b_{\gamma}$ in $\{b_i<b_{\alpha}<\cdots<b_{\beta}<b_{i+1}\}$ such that
its principal vector having  $\sum_{m\in [b_{\gamma}]}B_{12,m}E_m=0,$~i.e.,
$$~B_{12,m}=0 ,~~m\in[b_{\alpha}],\cdots,[b_{\beta}].$$

Let $i=1,j=2$ in the second equation of (\ref{xxx-1}) we have the following equation,
\begin{equation}\label{case2}
\sum_{m\neq 1,2}\frac{C_m^2(b_2-b_1)}{(b_m-b_1)^2(b_m-b_2)}+\sum_{m\neq1,2}\frac{B_{2m,1}^2}{(b_m-b_1)(b_m-b_2)}
+\sum_{m\neq 1,2}\frac{B_{2m,1}^2}{(b_m-b_2)^2}=0.
\end{equation}
Since $B_{12,m}=0$ for $m\in [b_{\alpha}],\cdots,[b_{\beta}]$ and  $(b_m-b_1)(b_m-b_2)>0$ for all index $m\not\in [b_{\alpha}],\cdots,[b_{\beta}]$, then all items in (\ref{case2}) are nonnegative, thus
\begin{equation}\label{case2-1}
C_m=B_{1m,2}=0,~~m\neq 1,2.
\end{equation}

For (\ref{case1-1}) and (\ref{case2-1}), $C_m=B_{1m,2}=0,~~m\neq 1,2.$ Next we prove that $$C_1C_2=C_3=\cdots=C_n=0.$$
Using (\ref{eqb2-1}) and  (\ref{xxx-12}), we have
\begin{equation}\label{case1-3}
A_{12}=\frac{2C_1C_2}{(b_1-b_2)^2},~~A_{1i}=A_{2i}=0,~~i\neq 1,2.
\end{equation}
Since $A_{12}, A_{11}, A_{22}$ are constant, From
$\sum_kA_{ij,k}\omega_k=dA_{ij}+\sum_kA_{ik}\omega_{kj}+\sum_kA_{kj}\omega_{ki}$ and (\ref{case1-3}), we have
\begin{equation}\label{case1-4}
\begin{split}
&A_{11,2}=\frac{2A_{12}C_1}{b_1-b_2},~~~~~~A_{22,1}=\frac{2A_{12}C_2}{b_2-b_1},\\
&A_{12,1}=\frac{(A_{11}-A_{22})C_2}{b_2-b_1},~~A_{12,2}=\frac{(A_{11}-A_{22})C_1}{b_2-b_1}.
\end{split}
\end{equation}
Combining with (\ref{equa1}), we have
$$(b_1-b_2)b_1C_2=(A_{22}-A_{11})C_2-2A_{12}C_1,~~(b_2-b_1)b_2C_1=(A_{11}-A_{22})C_1-2A_{12}C_2.$$
Combining the equation with (\ref{case1-3}), we can obtain the following equation:
$$(b_1-b_2)^2C_1C_2=\frac{-4C_1C_2(C_1^2+C_2^2)}{(b_1-b_2)^2},$$
which means that $C_1C_2=0.$ Combining (\ref{case1-1}) and (\ref{case2-1}), Lemma \ref{lemm4-1} is proved.
\end{proof}

\begin{lemma}\label{lemm4-2}
Under the basis (\ref{b31}), we have
$$(A_{ij})=diag(a_1,a_2,\cdots,a_n).$$
\end{lemma}
\begin{proof}
By Lemma \ref{lemm4-1},
$C_1C_2=C_3=\cdots=C_n=0,$ for $b_1<b_2<\cdots<b_s$. If $C_1=C_2=0$, then the M\"{o}bius form vanishes. By (\ref{eqb2-1}), we have $A_{ij}=0,~i\neq j,$ Thus Lemma \ref{lemm4-2} is proved.

If $C_1\neq 0$ and $C_2=\cdots=C_n=0$.  By (\ref{xxx-12}), for all simple principal vector fields $E_i,~1\leq i\leq s$, we have
\begin{equation}\label{ca1-5}
\begin{split}
&B_{1i,m}=0,~~m\neq 1,i.~~1\leq i\leq s,\\
&A_{im}=0,~~m\neq i,~~1\leq i\leq s.
\end{split}
\end{equation}
By (\ref{eqb2-1}), for multiple   principal vector fields $E_{\alpha},E_{\beta}$, we have
\begin{equation}\label{a1}
A_{\alpha\beta}=\frac{C_1B_{1\alpha,\beta}}{(b_1-b_{\alpha})(b_1-b_{\beta})}, ~~\alpha\neq\beta.
\end{equation}
If $[b_{\alpha}]=[b_{\beta}]$, then $B_{1\alpha,\beta}=0$ and
$$A_{\alpha\beta}=0, ~~[b_{\alpha}]=[b_{\beta}].$$
Thus to prove the lemma \ref{lemm4-2}, we need to prove $A_{\alpha\beta}=0,~~b_{\alpha}\neq b_{\beta}$ for any two distinct multiple M\"{o}bius
principal curvatures $b_{\alpha}, b_{\beta}$.

Now we fix two distinct multiple principal curvatures $b_{\alpha_0}$ and $b_{\beta_0}$.
We consider the matrix $$(T_{\alpha\beta})=(B_{1\alpha,\beta}),~~\alpha\in[b_{\alpha_0}],~~\beta\in[b_{\beta_0}].$$
The principal space
$$V_{\alpha_0}=Span\{E_i|i\in[b_{\alpha_0}]\},~~V_{\beta_0}=Span\{E_i|i\in[b_{\beta_0}]\}.$$
Using the singular value decomposition of a matrix, we can choose a basis $\{E_i|i\in[b_{\alpha_0}]\}$ in $V_{\alpha_0}$ and a basis
$\{E_i|i\in[b_{\beta_0}]\}$ in $V_{\beta_0}$ such that
\begin{equation}\label{ma}
(B_{1\alpha,\beta})=(T_{\alpha\beta})=\left(\begin{matrix}\Lambda&0\\
0&0\end{matrix}\right),
\end{equation}
where $\Lambda=diag(c_1,\cdots,c_t)$ is a diagonal matrix and $c_1,\cdots,c_t$ are nonzero constants.

If $B_{1\alpha,\beta}\in (T_{\alpha\beta})$ and $B_{1\alpha,\beta}= 0$.
From (\ref{a1}), we know that
$$A_{\alpha\beta}=0.$$
If $B_{1\alpha,\beta}\in (T_{\alpha\beta})$ and $B_{1\alpha,\beta}\neq 0$,
Using the second covariant derivatives of $B_{ij}$, defined by
$$
\sum_{l}B_{ij,kl}\omega_{l}=dB_{ij,k}+\sum_{l}B_{lj,k}\omega_{li}
+\sum_{l}B_{il,k}\omega_{lj}+\sum_{l}B_{ij,l}\omega_{lk},
$$
we get
\begin{equation}\label{bdd}
\begin{split}
&B_{1\alpha,\beta1}=\sum_{m\notin [b_{\alpha_0}],[b_{\beta_0}]}[\frac{B_{1m,\beta}B_{1m,\alpha}}{b_m-b_{\alpha_0}}+\frac{B_{1m,\beta}B_{1m,\alpha}}{b_m-b_{\beta_0}}],\\
&B_{1\alpha,1\beta}=\sum_{m\notin [b_{\alpha_0}],[b_{\beta_0}]}\frac{2B_{1m,\beta}B_{1m,\alpha}}{b_m-b_1}+2
\frac{C_1B_{1\alpha,\beta}}{b_1-b_{\beta_0}}+2\frac{C_1B_{1\alpha,\beta}}{b_1-b_{\alpha_0}}.
\end{split}
\end{equation}
Using Ricci identities $B_{1\alpha,1\beta}-B_{1\alpha,\beta1}=(b_1-b_{\alpha_0})R_{1\alpha1\beta}
=(b_1-b_{\alpha_0})A_{\alpha\beta}$,
we have
\begin{equation}\label{ara}
\begin{split}
(b_1-b_{\alpha_0})A_{\alpha\beta}=
&\sum_{m\notin [b_{\alpha_0}],[b_{\beta_0}]}B_{1m,\beta}B_{1m,\alpha}[\frac{2}{b_m-b_1}-\frac{1}{b_m-b_{\beta_0}}-\frac{1}{b_m-b_{\alpha_0}}]\\
&+\frac{2C_1B_{1\alpha,\beta}}{b_1-b_{\beta_0}}+\frac{2C_1B_{1\alpha,\beta}}{b_1-b_{\alpha_0}}.
\end{split}
\end{equation}
Similarly we consider matrix $(T_{\beta\alpha})=(B_{1\beta,\alpha}),~~\alpha\in[b_{\alpha_0}],~~\beta\in[b_{\beta_0}].$
We can obtain the equation
as (\ref{ara})
\begin{equation}\label{ara1}
\begin{split}
(b_1-b_{\beta_0})A_{\alpha\beta}=
&\sum_{m\notin [b_{\alpha_0}],[b_{\beta_0}]}B_{1m,\beta}B_{1m,\alpha}[\frac{2}{b_m-b_1}-\frac{1}{b_m-b_{\beta_0}}-\frac{1}{b_m-b_{\alpha_0}}]\\
&+\frac{2C_1B_{1\alpha,\beta}}{b_1-b_{\alpha_0}}+\frac{2C_1B_{1\alpha,\beta}}{b_1-b_{\beta_0}}.
\end{split}
\end{equation}
Noting that the right of the equation (\ref{ara}) and (\ref{ara1}) are equal, but the difference between the left of the equation (\ref{ara}) and (\ref{ara1}) is $(b_{\alpha_0}-b_{\beta_0})A_{\alpha\beta}$, which implies that $(b_{\alpha_0}-b_{\beta_0})A_{\alpha\beta}=0$. Since $b_{\alpha_0}\neq b_{\beta_0}$, so
$$A_{\alpha\beta}=0,~~\alpha\neq\beta.$$
Lemma \ref{lemm4-2} is proved.

If $C_2\neq 0$ and $C_1=C_3=\cdots=C_n=0$, using the same methods for $C_1\neq 0$ and $C_2=\cdots=C_n=0$, we can prove Lemma \ref{lemm4-2}.
\end{proof}

Now we prove Proposition \ref{prop3-2}. Under the frame (\ref{b31}),
\begin{equation*}
\begin{split}
&(A_{ij})=diag(a_1,a_2,\cdots,a_n),\\
&(B_{ij})=diag(b_1,\cdots,b_n)=diag(\bar{b}_1,\cdots,\bar{b}_s,\underbrace{\bar{b}_{s+1},\cdots,\bar{b}_{s+1}}_{m_{s+1}},
\cdots,\underbrace{\bar{b}_r,\cdots,\bar{b}_r}_{m_{r}}),
\end{split}
\end{equation*}
and $b_1<b_2<\cdots<b_s.$
By Lemma \ref{lemm4-1},
$C_1C_2=C_3=\cdots=C_n=0.$

First we assume $C_1\neq 0$, we will obtain a contradiction to prove $C=0$. Since $C_2=\cdots=C_n=0$, by (\ref{xxx-12}) and (\ref{a1}), we have
\begin{equation}\label{a2}
B_{1i,j}=0,~~i\neq j,~~1\leq i,j\leq n.
\end{equation}

Since $r>2$, there are at least three distinct M\"{o}bius principal curvatures $b_1, b_i, b_j$, where $b_i, b_j$ are possibly multiple or simple.
Now we fix index $i,j$ for $b_i\neq b_j$ and consider the three distinct M\"{o}bius principal curvatures $(b_1,b_i,b_j)$. Since $a_i, a_j$ are constant, from
$\sum_kA_{ij,k}\omega_k=dA_{ij}+\sum_kA_{ik}\omega_{kj}+\sum_kA_{kj}\omega_{ki}$ and (\ref{eqb1}), we have
\begin{equation*}
\begin{split}
&A_{ii,1}=0,~~~~~~A_{jj,1}=0,\\
&A_{1i,i}=\frac{(a_{1}-a_{i})C_1}{b_i-b_1},~~A_{1j,j}=\frac{(a_{1}-a_{j})C_1}{b_j-b_1}.
\end{split}
\end{equation*}
Combining with (\ref{equa1}), we have
\begin{equation}\label{ab1}
(b_i-b_1)b_i=(a_{1}-a_{i}),~~(b_j-b_1)b_j=(a_{1}-a_{j}),
\end{equation}
thus we have
\begin{equation}\label{ab1-1}
(b_j-b_i)(b_1-b_j-b_i)=(a_{j}-a_{i}).
\end{equation}
On the other hand, from (\ref{a2}) and (\ref{eqb2-2}) we have
$$R_{1j1j}=\frac{-C_1^2}{(b_1-b_j)^2},~~R_{1i1i}=\frac{-C_1^2}{(b_1-b_i)^2}.$$
by (\ref{equa4}),
$$R_{1j1j}-R_{1i1i}=b_1(b_j-b_i)+a_j-a_i=\frac{C_1^2(b_j-b_i)(2b_1-b_j-b_i)}{(b_1-b_j)^2(b_1-b_i)^2}.$$
Combining (\ref{ab1-1}), we have
$$(b_j-b_i)(2b_1-b_j-b_i)=\frac{C_1^2(b_j-b_i)(b_i+b_j-2b_1)}{(b_1-b_i)^2(b_1-b_j)^2}.$$
Since $1+\frac{C_1^2}{(b_1-b_i)^2(b_1-b_j)^2}\neq 0$, thus
\begin{equation}\label{ab2}
2b_1-b_i-b_j=0.
\end{equation}

Suppose there exists other multiple M\"{o}bius principal curvature $b_k$, where $b_k$ is possibly multiple or simple.
Now we fix index $i,k$ for $b_i\neq b_k$ and consider the three distinct M\"{o}bius principal curvatures $(b_1,b_i,b_k)$.
Similarly we can obtain the equation
\begin{equation}\label{ab3}
2b_1-b_i-b_k=0.
\end{equation}
The equation (\ref{ab2}) and (\ref{ab3}) imply that $b_j=b_k$, which is a contradiction. Therefore
there is no other M\"{o}bius principal curvature except $b_1,b_i,b_j$.

Since $B_{1i,j}=0$,  from (\ref{eqb2-2}) and (\ref{ab2}) we have
\begin{equation}\label{ar4}
R_{ijij}=\frac{-C_1^2}{(b_1-b_i)(b_1-b_j)}=\frac{C_2^2}{(b_1-b_j)^2}.
\end{equation}
From (\ref{ab2}), we have $R_{1i1i}-R_{1j1j}=0$, which implies
\begin{equation}\label{a7}
2a_1=a_i+a_j.
\end{equation}
Combining (\ref{ab1}) and (\ref{ar4}), we have
\begin{equation}\label{ar5}
R_{ijij}-R_{1j1j}=(b_i-b_j)(b_j-b_1)=\frac{2C_1^2}{(b_1-b_j)^2}.
\end{equation}
From (\ref{ab2}), we have $b_i<b_1<b_j$ or $b_j<b_1<b_i$, thus $(b_i-b_j)(b_j-b_1)<0$. Therefore the equation (\ref{ar4}) is a contradiction and thus $C_1=0$.

If $C_2\neq0$, then we can rearrange the order of basis in (\ref{b31}) $\{\tilde{E}_1,\tilde{E}_2,E_3,\cdots,E_n\}$ such that
$\tilde{E}_1=E_2, \tilde{E}_2=E_1$. Under the new basis $C_1\neq 0$, we can get a contradiction.

Thus Proposition \ref{prop3-2} is proved. Together with Proposition \ref{prop3-1} and Proposition \ref{prop3-3}, this complete the proof of the Vanishing Theorem~\ref{the1} as well as the Main Theorem~\ref{th11}.\\

{\bf Acknowledgements:} The first author is  supported by NSFC  grant 12071028.
The second author and the third author are  supported by NSFC grant 11831005. The fourth author is  supported by NSFC  grant 11971107.


\begin{thebibliography}{11}
\bibitem{ca1} E. Cartan, {\sl Sur des familes remarquables d'hypersurfaces isoparam\'{e}triques dans les espace sph\'{e}riques}, Math. Z., 45(1939),335-367.
\bibitem{ce} T. E. Cecil, {\sl Lie Sphere Geometry: With applications to Submanifolds}, Springer, New York, 1992.
\bibitem{chenli} Y. Y. Chen, X. Ji, T. Z. Li, {\sl M\"{o}bius homogeneous hypersurfaces with one simple principal curvature in $\mathbb{S}^{n+1}$,}
Acta Math. Sin. (Engl.ser.), Vol.36, (2020),1001-1013.
\bibitem{chi} Q. S. Chi, {\sl Isoparametric hypersurfaces with four principal curvatures,III}, J.Diff. Geom., 94(2013),469-504.
\bibitem{chi1} Q. S. Chi, {\sl The isoparametric story: A Heritage of \'{E}lie Cartan}, Proceedings of the International Consortium of Chinese Mathematicians 2018, 197-260, Int. Press, Boston, MA, (published in 2020). Online preprint arXiv:2007.02137.
\bibitem{di} A. J. Di Scala, {\sl Minimal homogeneous submanifolds in Euclidean spaces}, Ann. Glob. Anal. and Geom., Vol.21,(2002)15-18.
\bibitem{olmos} A. J. Di Scala, Carlos Olmos, {\sl The geometry of homogeneous submanifolds of hyperbolic space}, Math. Z., 237(2001), 199-209.
\bibitem{Xiang} W. Y. Hsiang, H. B. Lawson, Jr., {\sl Minimal submanifolds of low cohomogeneity,} J.
Diff. Geom. 5 (1971), 1-38.
\bibitem{jl} X. Ji, T. Z. Li, {A note on compact M\"{o}bius homogeneous submanifolds in $\mathbb{S}^{n+1}$},
Bull. Korean Math. Soc. 56 (2019),681-689.
\bibitem{LiT} T. Z. Li, {\sl M\"{o}bius homogeneous hypersurfaces with three distinct principal curvatures in $\mathbb{S}^{n+1}$},
Chinese Annals of Mathematics, Ser.B, Vol. 38, (2017)1131-1144.
\bibitem{lim} T. Z. Li, X. Ma, C. P. Wang, {\sl M\"{o}bius homogeneous hypersurfaces with two distinct principal curvatures in $\mathbb{S}^{n+1}$},
Arkiv f\"{o}r Matematik,  51(2013),315-328.
\bibitem{lim2} T. Z. Li, C. P. Wang, {\sl Classification of M\"{o}bius homogeneous hypersurfaces in $5$-dimensional sphere},
Houston Journal of Mathematics, Vol. 40(2014)1127-1146.
\bibitem{lim3} T. Z. Li, J. Qing, C. P. Wang, {\sl M\"{o}bius  Curvature, Laguerre Curvature and Dupin Hypersurface}, Adv. Math., Vol. 311(2017)
249-294.
\bibitem{ma} Xiang Ma, Franz Pedit, Peng Wang, {\sl M\"{o}bius homogeneous Willmore $2-$spheres}, Bull. Lond. Math. Soc., Vol.50, (2018), 509-512.
\bibitem{mi} Reiko Miyaoka, {\sl Isoparametric hypersurfaces with $(g,m)=(6,2)$}, Ann. of Math. (2) 177 (2013), no. 1, 53-110. ~errata, Ann. of Math. (2) 183 (2016), no. 3, 1057-1071.
\bibitem{mu1} H. F. M\"{u}nzner, {\sl Isoparametrische Hyperfl\"{a}chen in Sph\"{a}ren}, Math. Ann., 251,(1980), 57-71.
\bibitem{mu2} H. F. M\"{u}nzner, {\sl Isoparametrische Hyperfl\"{a}chen in Sph\"{a}ren II: \"{U}ber die Zerlegung der Sph\"{a}re in Ballb\"{u}ndel},
Math. Ann., 256, (1981), 215-232.

\bibitem{o} B. O'Neil, {\sl Semi-Riemannian Geometry}, Academic Press, New York,1983.
\bibitem{s} R. Sulanke, {\sl M\"{o}bius geometry V: Homogeneous surfaces in the M\"{o}bius space $S^3$},
 Topics in differential geometry, Vol.I,II (Debrecen, 1984),
1141-1154.
\bibitem{ta} R. Takagi, T. Takahashi, {\sl On the principal curvatures of homogeneous hypersurfaces
in a unit sphere}, Differential Geometry, in honor of K. Yano, Kinokuniya,
Tokyo, 1972, 469-481.
\bibitem{w} C. P. Wang, {\sl Moebius geometry of submanifolds in
$S^n$}, Manuscripta Math., 96(1998),517-534.
\bibitem{w1} C. P. Wang, {\sl M\"{o}bius geometry for hypersurfaces in $S^4$}, Nagoya Math. J., Vol. 139(1995),1-20.
\bibitem{wx} C. P. Wang, Z. X. Xie, {\sl Classification of M\"{o}bius homogenous surfaces in $S^4$}, Ann. Global Anal. Geom. 46(2014), no. 3, 241-257.
\end{thebibliography}
\end{document}